\documentclass[11pt]{amsart}
\usepackage{amsmath,amssymb,latexsym,tikz}
\usepackage{mathrsfs,amsfonts}
\usepackage[colorlinks=true,urlcolor=blue,
citecolor=red,linkcolor=blue,linktocpage,pdfpagelabels,
bookmarksnumbered,bookmarksopen]{hyperref}
\usepackage[english]{babel}

\usepackage[left=3cm,right=3cm,top=3cm,bottom=3cm]{geometry}

\usepackage[hyperpageref]{backref}
\usepackage{mathrsfs}
\usepackage{lmodern}

\numberwithin{equation}{section}

\newtheorem{theorem}{Theorem}[section]
\newtheorem{lemma}[theorem]{Lemma}

\newtheorem{corollary}[theorem]{Corollary}
\theoremstyle{definition}

\newtheorem{remark}[theorem]{Remark}

\newcommand{\N}{{\mathbb N}}

\newcommand{\R}{{\mathbb R}}

\newcommand{\eps}{\varepsilon}

\newcommand{\cal}{\mathcal}

\def\XXint#1#2#3{{\setbox0=\hbox{$#1{#2#3}{\int}$ }
\vcenter{\hbox{$#2#3$ }}\kern-.6\wd0}}

\title{Asymptotics for optimizers of the fractional Hardy-Sobolev inequality}

\author[S.A.  Marano]{S.~A.\,  Marano}
\address[S.A.  Marano]{Dipartimento di Matematica e Informatica
\newline\indent
Universit\`a degli Studi di Catania
\newline\indent
Viale A. Doria 6 I-95125 Catania, Italy}
\email{marano@dmi.unict.it}

\author[S. Mosconi]{S.~J.~N.\,  Mosconi}
\address[S. Mosconi]{Dipartimento di Matematica e Informatica
\newline\indent
Universit\`a degli Studi di Catania
\newline\indent
Viale A. Doria 6 I-95125 Catania, Italy}
\email{mosconi@dmi.unict.it}

\subjclass[2010]{6E35, 35B40, 49K22}

\keywords{Fractional Hardy-Sobolev inequality, concentration-compactness, decay estimates, fractional $p$-Laplacian}

\thanks{The authors were partially supported by Gruppo Nazionale per l'Analisi Matematica, la Probabilit\`a e le loro Applicazioni (INdAM)}

\begin{document}
\begin{abstract}
The existence of optimizers $u$ in the space $\dot{W}^{s,p}(\R^N)$, with differentiability order $s\in\  ]0,1[$, for the Hardy-Sobolev inequality is established through concentration-compactness. The asymptotic behavior $u(x)\simeq |x|^{-\frac{N-ps}{p-1}}$ as $|x|\to+\infty$ and the summability information $u\in \dot{W}^{s,\gamma}(\R^N)$ for all $\gamma\in \ ]\frac{N(p-1)}{N-s}, p]$ are then obtained. Such properties turn out to be optimal when $s\to 1^-$, in which case optimizers are explicitly known.
\end{abstract}
\maketitle
\section{Introduction and main results}
This paper deals with the scale-invariant, nonlocal functional inequality 
\begin{equation}\label{HS}
\left(\int_{\R^N} \frac{|u|^q}{|x|^\alpha}\, dx\right)^{\frac{1}{q}}\leq C\left(\int_{\R^N\times\R^N}\frac{|u(x)-u(y)|^p}{|x-y|^{N+ps}}\, dx\, dy\right)^{\frac{1}{p}}
\end{equation}
for some constant $C>0$. Here, $N> \alpha \geq 0$, $q\geq p\geq 1$, and $s\in \ ]0,1[$ are determined by scale invariance. In order that $C$ be finite, one can write \eqref{HS} for $u_\lambda(x):=u(\lambda x)$, deducing
\begin{equation}\label{scalingrelation}
\frac{N-\alpha}{q}=\frac{N-ps}{p}.
\end{equation}
So, the constant $C$ depends on $N, p, s,\alpha$, namely $C:=C(N, p, s,\alpha)$. Further, since $q\geq p$, from \eqref{scalingrelation} we immediately infer
\[
0\leq \alpha\leq ps<N.
\]
If $q=p$ and $\alpha=ps<N$ then the classical Hardy fractional inequality is recovered. For $q=p^*$, where
\begin{equation}\label{pstar}
p^*:=\frac{Np}{N-ps},
\end{equation}
and $\alpha=0$ we obtain the fractional Sobolev inequality. The general expression \eqref{HS} is usually called fractional Hardy-Sobolev inequality; see \cite{Mazya}. 

It is well known that \eqref{HS} with $q=p>1$, whence $\alpha=ps$, does not admit optimizers and indeed the concentration-compactness method fails. More precisely, letting
\[
[u]_{s,p}^p:=\int_{\R^N\times\R^N}\frac{|u(x)-u(y)|^p}{|x-y|^{N+ps}}\, dx\, dy
\]
and
\begin{equation}\label{I}
I_\lambda:=\inf\left\{[u]_{s,p}^p:\int_{\R^N} \frac{|u|^q}{|x|^\alpha}\, dx=\lambda\right\}, \qquad \lambda>0,
\end{equation}
after scaling one has
\begin{equation}\label{Ilambda}
I_\lambda=\lambda^{\frac{p}{q}} I_1, 
\end{equation}
and thus the strict subadditivity condition $I_{\lambda+\mu}<I_\lambda+I_\mu$, which represents the main tool of concentration-compactness, holds only if $q>p$. In such a case, existence of optimizers has mostly been taken as granted thanks to \cite[Remark I.6]{Lions1}. For $p=2$, a full proof has been done in \cite{PP}, with $\alpha=0$, exploiting a refined version of Sobolev's embedding (obtained via Morrey spaces), and in \cite{Y}, where $\alpha\in [0, ps[$. We will instead establish the existence of optimizers for general $p>1$ through the original Lions' approach, accordingly considering (as a natural non-local counterpart of $|\nabla u|^p(x)$) the energy-density function
\begin{equation}\label{defDsup}
|D^s u|^p(x):=\int_{\R^N}\frac{|u(x)-u(x+h)|^p}{|h|^{N+ps}}\, dh.
\end{equation}
As it turns out, the proof is quite involved at times, mainly due to the fact that non-local interactions arise when one analyzes dichotomy and/or concentration. They are typical of non-local problems and, to treat them, we will use estimates having no analogue in the local framework.

Once existence is achieved, standard rearrangement inequalities ensure that the minimizers are radially monotonic. Hence, a natural conjecture is whether the family of minimizers consists of constant multiples, translations, and dilations of the function
\begin{equation}\label{talentiane}
U(x):=\frac{1}{(1+|x|^{\frac{p-\alpha/s}{p-1}})^{\frac{N-sp}{p-\alpha/s}}},\quad x\in\R^N,
\end{equation}
which coincides with the classical Aubin-Talenti function provided $s=1$, $\alpha=0$. Such a conjecture has been proved in \cite{GY} when $s=1$, $p>1$, $\alpha\in [0, p[$ through Bliss inequality, and in \cite{CLO} if $s\in\ ]0, 1[$, $p=2$, $\alpha=0$. However, up to now, the explicit form of optimizers is not known for general $s\in \ ]0, 1[$ and $p\neq 2$. Notice that, contrary to the local case, where a simple ODE argument applies, when $s\in \ ]0, 1[$ it is not even clear how to show that \eqref{talentiane} at least solves the corresponding Euler-Lagrange equation.

To attack the problem of finding minimizers in \eqref{I}, a first step might be to check compatibility of the {\em a priori} asymptotic behavior of minimizers with the one exhibited by \eqref{talentiane}, i.e., 
\[
U(R)\simeq\frac{1}{R^{\frac{N-ps}{p-1}}},\quad \nabla U(R)\simeq \frac{1}{R^{\frac{N-ps}{p-1}+1}},\quad R\to+\infty.
\]
Concerning the first estimate, we will show that any minimizer $u$ actually obeys the same asymptotics as $U$. Relevant arguments are patterned after those of \cite{BMS}, where $\alpha=0$. The second estimate clearly requires much higher regularity than the natural one for $u$, which seems out of reach when $s$ is very small. Accordingly, we will consider an appropriate weaker version that, if $s=1$, reads as
\begin{equation}
\label{decloc}
\nabla U \in L^\gamma(\R^N)\quad \Leftrightarrow\quad \gamma\in \ ]\frac{N(p-1)}{N-1}, p].
\end{equation}
Since one is interested in {\em decay} estimates, the {\em lowest possible} summability exponent of $\nabla U$ has to be sought out. For $s\in \ ]0, 1[$, we first observe that the asymptotic behavior 
\[
v(R)\geq \delta\,  R^{-\frac{N-ps}{p-1}}, \qquad \delta>0, \quad R\geq 1,
\]
combined with  Hardy's inequality (which holds for any $\gamma\in ]0, N/s[$; see \cite[Theorem 1.1]{D})
\[
[v]_{s,\gamma}^\gamma\geq \frac{1}{C^\gamma}\int_{\R^N} \frac{|u|^\gamma}{|x|^{\gamma s}}\, dx\geq \frac{\delta^\gamma}{C^\gamma}\int_{\R^N\setminus B_1}|x|^{-\gamma\frac{N-s}{p-1}}\, dx
\]
produces the information
\[
[v]_{s, \gamma}<+\infty \quad \Rightarrow\quad \gamma>\frac{N(p-1)}{N-s}.
\]
This obviously applies to the function $U$ defined in \eqref{talentiane} and, by the obtained asymptotics, to any  optimizer $u$ of \eqref{I}  as well. Theorem \ref{MT} below ensures that the opposite implication holds true both for every minimizer $u$ and for $U$. The condition $\gamma \in\ ] \frac{N(p-1)}{N-s}, p]$ is thus optimal in the decay sense. 

Our motivation does not rely only on the asymptotic compatibility of the conjectured form \eqref{talentiane} of minimizers. Indeed, the decay estimate
\begin{equation}\label{dec}
u(x)\simeq \frac{1}{|x|^{\frac{N-ps}{p-1}}}, \quad |x|\to+\infty
\end{equation}
has proved to be useful for treating the nonlinear, critically perturbed, eigenvalue problem
\[
\begin{cases}
(-\Delta_p)^s u=\lambda |u|^{p-2}u+ |u|^{p^*-2}u&\text{in $\Omega$},\\
u\equiv 0&\text{in $\R^N\setminus \Omega$}
\end{cases}
\]
analogous to the Brezis-Nirenberg one. Here, $p^*$ is given by \eqref{pstar} while $(-\Delta_p)^s$ denotes the fractional $p-s$ Laplacian, defined as the differential of $u\mapsto \frac 1 p [u]_{s,p}^p$. Very recently, existence results have been obtained in \cite{MPSY} chiefly through cutoff and rescaling of solutions to \eqref{I}, using only the scaling properties of $u_\eps$ and the pointwise decay \eqref{dec}.

However, when we deal with more general operators of mixed order, as the $(p,q)$-Laplacian, a precise estimate at other, less natural, differentiability scales (e.g., $\|\nabla u\|_{q}$ if $q<p$) is essential. Such information has been achieved in \cite{DH}, provided $s=1$, $\alpha=0$, through the explicit form of minimizers for \eqref{I}. The corresponding results have been extensively exploited to treat mixed critical problems; see \cite{CMP, YY} and the references therein. We therefore plan to apply the $s$-derivative decay estimate established below to analogous mixed fractional order problems in future works.

The results of the paper can be summarized as follows.
\begin{theorem}\label{MT}
Let $N>ps$, $q>p$, and $\alpha\in  [0, ps[$ satisfy \eqref{scalingrelation}. Then:
\begin{itemize}
\item  Problem \eqref{I} has a minimizer.
\item Every minimizer $u$ is of constant sign, radially monotone, and fulfills
\begin{equation}\label{as1}
\frac{1}{C|x|^{\frac{N-ps}{p-1}}}\leq |u(x)|\leq \frac{C}{|x|^{\frac{N-ps}{p-1}}},\quad |x|\geq 1,
\end{equation}
for some constant $C:=C(N, p, s, \alpha, u)$, as well as 
\begin{equation}\label{as2}
\int_{\R^N\times \R^N}\frac{|u(x)-u(y)|^\gamma}{|x-y|^{N+\gamma s}}\, dx\, dy<+\infty\quad \forall\,\gamma\in\  \left] \frac{N(p-1)}{N-s}, p\right].
\end{equation}
\item Estimates \eqref{as1}--\eqref{as2} hold for the function $U$ defined in \eqref{talentiane} and thus for any translation, multiple, and rescaling of it.
\end{itemize}
\end{theorem}
\vskip5pt
{\bf Sketch of proof}.
Since the main novelty is \eqref{as2}, it may be instructive to look at a simple proof of \eqref{decloc} {\em without knowing the minimizer's explicit form}.

When $\alpha=0$, nonnegative minimizers satisfy the Euler-Lagrange equation
\begin{equation}\label{pik}
-\Delta_p u=  c\, u^{p^*-1}=:f.
\end{equation}
Then a variant of the Strauss lemma for radially decreasing functions yields the decay estimate $u(x)\leq C|x|^{-\frac{N-p}{p-1}}$, with large $|x|$, provided $f\in L^1(\R^N)$ (a nontrivial fact at the global level). 

To prove \eqref{decloc}, we first decompose $u$ in its horizontally dyadic components, given by slicing $u$ at heights $u(2^i)$:
\[
u=\sum_{i=0}^{+\infty} u_i, \quad\mbox{where}\quad  0\leq u_i\leq u(2^{i-1})\, \chi_{B_{2^i}}\, ,\quad i\geq 1.
\]
We avoid here more involved arguments and assume $\gamma\geq 1$. Thus, on account of the triangle inequality, it suffice to estimate $\|\nabla u_i\|_{L^\gamma}$ separately. As ${\rm supp}(u_i)\subseteq B_{2^i}$ and $\gamma\leq p$, H\"older's inequality gives 
\begin{equation}\label{nMin}
\|\nabla u_i\|_{L^\gamma}\leq |B_{2^i}|^{1-\frac{\gamma}{p}}\, \|\nabla u_i\|_{L^p}\leq C\, 2^{iN(1-\frac{\gamma}{p})}\, \|\nabla u_i\|_{L^p}.
\end{equation}
Now, since $\nabla u_i=\nabla u$ a.e. in $B_{2^i}\setminus B_{2^{i-1}}$, $\nabla u_i=0$ a.e. in $B_{2^{i-1}}$, and $u_i\leq u(2^{i-1})$, testing \eqref{pik} with $u_i$ we obtain
\[
\|\nabla u_i\|_{L^p}^p=\langle -\Delta_p u, u_i\rangle=\int_{\R^N} f \, u_i\, dx\leq \|f\|_{L^1}\, u(2^{i-1})\leq C\, \|f\|_{L^1}\, 2^{-i\frac{N-p}{p-1}},
\]
where the decay estimate has been used in the last inequality. Finally, raise to the $1/p$-power and insert inside \eqref{nMin}, to achieve
\[
\|\nabla u\|_{L^\gamma}\leq \sum_{i=0}^{+\infty}\|\nabla u_i\|_{L^\gamma}\leq C\, \|f\|_{L^1}^{\frac{1}{p}}\sum_{i=0}^{+\infty}2^{iN(1-\frac{\gamma}{p})}\, 2^{-i\frac{N-p}{p(p-1)}},
\]
which is finite as long as $\gamma>N(p-1)/(N-1)$.
\vskip5pt
{\bf Difficulties}.
Two main issues arise in trying to reproduce the previous proof for the fractional case $s\in \ ]0, 1[$. 

The first one is technical, because in order to obtain the best summability lower bound $\gamma>N(p-1)/(N-s)$ we may have to deal with exponents $\gamma$  less than $1$. Not only the triangle inequality fails in such a case but, more importantly, there does not seem to exist a satisfying interpolation theory for the concrete spaces $W^{s,\gamma}(\R^N)$ (contrary to the interpolation theory for Besov-Lizorkin spaces, mainly due to Peetre in the low summability case).

The second one, however, is substantial. Inequality \eqref{nMin} evidently implies the natural embedding $W^{1,p}(B_R)\hookrightarrow W^{1,\gamma}(B_R)$ for $p\geq \gamma$, which is actually {\em false} in the fractional case, as an example of Mironescu and Sickel \cite{MiS} shows. Indeed, for any $s\in\  ]0, 1[$ and any $p>\gamma\geq 1$, the space $W^{s,p}(B_R)$ is never a subset of $W^{s,\gamma}(B_R)$. This forces a weakening of \eqref{nMin} through an interpolation inequality and a higher differentiability estimate in Besov spaces for the minimizers.

\vskip5pt
{\large{\bf Outline of the paper}}.
Let us finally discuss the structure of the paper. Section 2 is devoted to  framework and known tools that will be employed for proving Theorem \ref{MT}.

The existence part is performed in Section 3 via concentration-compactness. We remark that other approaches are available, as, e.g., the one of Lieb \cite{lieb}, based on rearrangements. However, Lion's technique seems viable to treat (more general) situations where rearrangement is not available. In developing the concentration-compactness scheme for nonlocal problems, two main difficulties arise. Ruling out dichotomy for minimizing sequences is quite delicate, since splitting a function through cutoffs gives rise to nonlocal effects, which have to be precisely quantified. This is done by observing that the smallness of $|D^su|^p$ as per \eqref{defDsup}, contrary to the local case, entails strong {\em global} information on $u$; for instance, if $|D^s u|^p$ vanishes at some point then $u$ must be constant. The quantitative estimate needed to rule out dichotomy is Lemma \ref{lemmaint} below. This technical result deals with the loss of compactness due to translation (precisely in the dichotomy case), while the other difficulty lies in the nonlocal effects stemming from the loss of compactness due to dilations, i.e., concentration.  However, in this respect, the relevant argument has been derived in \cite[Theorem 2.5]{MS} and we refer to the discussion therein for further details.

Section 4 involves some general regularity results for the model equation $(-\Delta_p)^su=f$ on the entire space. We will be concerned with both summability and differentiability estimates. The former are more or less already available in the literature, although the fact that we work on the whole $\R^N$ requires some care. The higher differentiability of solutions to the model equation has been treated only recently in \cite{BL} in the superquadratic case, assuming various {\em differentiability} hypotheses on the forcing term $f$. We are then going to refine and extend the techniques of \cite{BL} to obtain a higher Besov regularity result solely under suitable {\em summability} assumptions on $f$; see Lemma \ref{regest}. 

Finally, Section 5 contains the proof of Theorem \ref{MT}. As outlined before, the main tool is an $L^1$ estimate of the forcing term of the corresponding Euler-Lagrange equation for minimizers, which readily implies the decay estimate \eqref{as1}. Relevant arguments are patterned after \cite{BMS}. To show \eqref{as2}, we decompose a given minimizer into its horizontally dyadic components and evaluate the corresponding $W^{s, \gamma}$ terms separately. An interpolation inequality (cf. Lemma \ref{intbesov}), together with the slightly better differentiability properties of minimizers, ensure that the $W^{s,\gamma}$ energy of the dyadic components can (almost entirely) be controlled in terms of their $W^{s, p}$ norm. Exploiting \eqref{as1} to estimate them via the Euler-Lagrange equations, produces \eqref{as2}. 
\section{Preliminary material}

Let us first fix some notation. If $p> 1$, we put $p':=p/(p-1)$ while $p':=\infty$ when $p=1$. Denote by $B_r(x)$ the open ball of center $x$ and radius $r>0$ in $\R^N$. If the center is not specified then it is to be understood as zero, i.e., $B_r:=B_r(0)$. Given any Lebesgue measurable set $E\subseteq \R^N$, $\chi_E$ denotes its indicator function, $|E|$ its Lebesgue measure, and $E^c:=\R^N\setminus E$. Finally, $\omega_N:=|B_1(0)|$. 

The {\em symmetric-decreasing rearrangement} of a measurable function $u:\R^N\to \R_+$ is, by definition, the radial function  $u^*(x)=u^*(|x|)$ such that $u^*$ is non-increasing, right continuous with respect to $r:=|x|$, and 
\[
|\{u^*>t\}|=|\{u>t\}|\quad \text{for a.e. $t\geq 0$}.
\]
From now on, the dimension $N\geq 1$ will be fixed, $s\in \ ]0,1[$, and $p\geq 1$ will satisfy $ps<N$. Moreover, $p^*$ denotes the fractional critical exponent, namely $p^*:=Np/(N-ps)$.

Let $u:\R^N\to\R$ be measurable. We say that $u$ {\em vanishes at infinity} if
 \begin{equation}\label{vi}
|\{|u|>a\}|<+\infty\quad \text{for all $a>0$}.
\end{equation} 
Define 
\[
|D^s u|^p(x)=\int_{\R^N}\frac{|u(x)-u(y)|^p}{|x-y|^{N+ps}}\, dy,\quad x\in\R^N.
\]
Elementary inequalities ensure that for all measurable $v,w:\R^N\to\R$ one has
\begin{equation}\label{Liebn}
\int_{\R^N}|D^s(v\, w)|^p\, dx\leq 2^{p-1}\int_{\R^N} \left(|v|^p\, |D^s w|^p+ |w|^p\, |D^s v|^p\right) dx.
\end{equation}
Next, set
\[
[u]_{s,p}:=\left(\int_{\R^N\times\R^N}\frac{|u(x)-u(y)|^p}{|x-y|^{N+ps}}\, dx\, dy\right)^{\frac{1}{p}}
\]
in addition to
\[
\|u\|_{\alpha, q}:=\left(\int_{\R^N} \frac{|u|^q}{|x|^\alpha}\, dx\right)^{\frac{1}{q}},
\]
where $\alpha\in [0, N[$ and $q\geq 1$. Obviously, $\|u\|_q:=\|u\|_{0, q}$. The symbol $\dot{W}^{s,p}(\R^N)$ denotes the homogeneous Sobolev space 
\[
\dot{W}^{s,p}(\R^N):=\{u: \text{$u$ is measurable, \eqref{vi} holds, and}\  [u]_{s,p}<+\infty\},
\]
while for any $\Omega$ open (not necessarily bounded) subset of $\R^N$ we let
\[
W^{s,p}_0(\Omega):=\{u\in \dot{W}^{s,p}(\R^N): u\equiv 0 \, \text{ a.e. in $\Omega^c$}\}.
\]
Clearly, $\dot{W}^{s,p}(\R^N)=W^{s,p}_0(\R^N)$. If $p>1$ then both ${\dot W}^{s,p}(\R^N)$ and $W^{s,p}_0(\Omega)$ are reflexive Banach spaces with respect to the norm $[\cdot ]_{s,p}$. Moreover, $C^\infty_c(\R^N)$ is a dense subspace of them provided $\Omega=\R^N$ or $\Omega$ is bounded and smooth. 

$\dot{W}^{s,p}(\R^N)$ falls inside the wider class of Besov spaces, whose definition we now recall. For any $h\in \R^N\setminus\{0\}$ and measurable $g:\R^N\to \R$, put
\[
g_h(x):=g(x+h),\quad  \delta_h g(x):=g_h(x)-g(x),\quad  \delta^2_h g:=\delta_{-h}(\delta_h g)=\delta_h(\delta_{-h} g).
\]
Given $1\leq p<+\infty$, $1\leq q\leq+\infty$, and $s\in \ ]0, 2[\ \setminus \{1\}$, the classical (homogeneous) Besov space is
\[
\dot B^{s}_{p, q}(\R^N):=\left\{u: \text{\eqref{vi} holds and} \ [u]_{B^{s}_{p,q}}^p:=\int_{\R^N}\left(\int_{\R^N}\left|\frac{\delta^2_hu(x)}{|h|^s}\right|^p dx\right)^{q/p}.\frac{dh}{|h|^N}<+\infty\right\}
\]
when $q<+\infty$, while 
\[
\dot B^{s}_{p, \infty}(\R^N):=\left\{u: \text{\eqref{vi} holds and} \ [u]_{B^{s}_{p,\infty}}^p:=\sup_{h\neq 0}\int_{\R^N}\left|\frac{\delta^2_hu(x)}{|h|^s}\right|^p dx<+\infty\right\}
\]
for $q=+\infty$. Recall that $[\cdot ]_{B^{s}_{p, q}}$ turns out to be a norm on $\dot{B}^{s}_{p, q}(\R^N)$ and that $(\dot{B}^{s}_{p, q}(\R^N),[\cdot ]_{B^{s}_{p, q}})$ is complete. For larger values of $s$, which we don't need in the sequel, the norm actually involves higher order differences $\delta^m_hu$, where $m>s$. Finally, if $s\in \ ]0, 1[$ then 
$$\dot B^{s}_{p, p}(\R^N)=\dot W^{s,p}(\R^N),$$
and the respective norms are equivalent; see, e.g., \cite[Chapter 10]{Mazya}. In such a case, by a simply changing variables, one has
\[
[u]_{s,p}^p=\int_{\R^N}\left\|\frac{\delta_hu}{|h|^s}\right\|_p^p \frac{dh}{|h|^N}.
\]

Given an arbitrary nonempty open set $\Omega\subseteq \R^N$, the functional $u\mapsto \frac 1 p [u]_{s,p}^p$, $u\in W^{s,p}_0(\Omega)$,  turns out to be convex and differentiable provided $p>1$. Its differential $(-\Delta_p)^su$ at any point $u$ lies in
$$W^{-s,p'}(\Omega):=\big(W^{s,p}_0(\Omega)\big)^*.$$
One clearly has $L^{(p^*)'}(\R^N)\hookrightarrow W^{-s, p'}(\R^N)$, because 
$$W^{s,p}_0(\Omega)\hookrightarrow\dot W^{s,p}(\R^N)\hookrightarrow L^{p^*}(\R^N).$$
Thus, the equation $(-\Delta_p)^s u=f$ makes sense (weakly) in $\Omega$ for all $f\in L^{(p^*)'}(\R^N)$. Here, we will be concerned with \emph{more general right-hand sides}. Precisely, let $f\in L^{1}_{\rm loc}(\R^N)$. We say that $u\in W^{s,p}_0(\Omega)$ is a weak solution of $(-\Delta_p)^su=f$ provided
\begin{equation}\label{locweak}
\langle (-\Delta_p)^s u, \varphi\rangle=\int_{\R^N} f\, \varphi\, dx\quad \forall\,\varphi\in W^{s,p}_0(\Omega)
\mbox{ such that } f\,\varphi\in L^1(\R^N),
\end{equation}
where $\langle \ , \ \rangle$ denotes the duality pairing between $W^{-s, p'}(\Omega)$ and $W^{s,p}_0(\Omega)$. Every $\varphi\in W^{s,p}_0(\Omega)$ fulfilling $f\,\varphi\in L^1(\R^N)$ will be called a \emph{suitable test function}. If $\Omega=\R^N$ and
$\varphi\in\dot{W}^{s,p}(\R^N)$ is bounded and has a finite measure support then $\varphi$ turns out to be a suitable test function.

Let us now recall some facts about the space $\dot{W}^{s,\gamma}(\R^N)$ with $\gamma\in\ ]0, 1[$. One can introduce again the vector space
\[
\dot{W}^{s, \gamma}(\R^N)=\{u: \text{$u$ is measurable, \eqref{vi} holds and } [u]_{s,\gamma}<+\infty\},
\]
but it isn't a Banach space and, for sufficiently small $\gamma>0$, its elements may fail to be locally integrable and therefore to be distributions. On the other hand, the Besov spaces $\dot{B}^s_{\gamma, q}(\R^N)$ for $s,q,\gamma>0$ can still be defined via suitable decay properties of the Littlewood-Paley decomposition, giving rise to a space of distributions (which may contain singular measures). However, unless $\gamma>\frac{N}{N+s}$, it is no longer true that $\dot{W}^{s, \gamma}(\R^N)=\dot{B}^{s}_{\gamma, \gamma}(\R^N)$; see \cite[Section 2.5.12]{T} for more details.

Regarding the basic properties of problem \eqref{HS}, we start by pointing out that, in order to seek optimizers in \eqref{HS}, radial functions suffice.

\begin{lemma}\label{radiality}
Suppose $u\in\dot W^{s, p}(\R^N)$ is a minimizer of \eqref{I}. Then $u$ turns out to be radially non-increasing around some point, which is zero if $\alpha>0$.
\end{lemma}

\begin{proof}
An easy computation based on \eqref{Ilambda} ensures that $u$ realizes the Rayleigh quotient
\[
{\cal R}:=\inf\left\{\frac{[v]_{s,p}^p}{\|v\|_{\alpha, q}^{p}}: v\in\dot{W}^{s,p}(\R^N)\setminus \{0\}\right\}.
\]
Let $u^*$ be the symmetric-decreasing rearrangement of $u$. Thanks to Theorem 3.4 of  \cite{LiebLoss} we have
\[
\|u\|_{\alpha, q}\leq \|u^*\|_{\alpha, q},
\]
because $x\mapsto 1/|x|^\alpha$ coincides with its symmetric-decreasing rearrangement. If $\alpha:=0$ then equality always holds, while when $u\neq u^*$ and  $\alpha>0$ the inequality turns out to be strict. Moreover, the P\'olya-Szeg\"o principle \cite[Theorem 9.2]{AL} gives
 \[
 [u^*]_{s,p}\leq [u]_{s,p},
 \]
with strict inequality once $u$ is not a translation of $u^*$. The last assertion is peculiar to the nonlocal nature of $[u]_{s,p}$; see \cite[Theorem A1]{FS}.
\end{proof}
Let us also observe that inequalities \eqref{HS} stem from the borderline Hardy inequality, namely \eqref{HS} written for
$\alpha:=ps$ and $q:=p$. The corresponding best constant, say $C_H$, has been explicitly computed in \cite[Theorem 1.1]{FS}. The next result basically is folklore.
\begin{lemma}
Let $C_H:=C(N, p, s, ps)$ in \eqref{HS}. Then 
\[
C(N, p, s, \alpha)\leq C_H\left(\frac{N\omega_N}{N-ps}\right)^{\frac{1}{p}\frac{\alpha-ps}{N-\alpha}}.
\]
\end{lemma}

\begin{proof}
Pick any $u\in C^\infty_c(\R^N)$. On account of Lemma \ref{radiality}, we may assume $u=u(r)$ both nonnegative and radially non-increasing. It is known \cite{HLP} that if $\lambda\geq 1$ and  $h:[0, +\infty)\to \R^+_0$ is non-increasing then
\[
\int_0^{+\infty} h(t)^\lambda\, t^{\lambda-1} dt= \int_0^{+\infty} (h(t)t)^{\lambda-1}h(t) dt\leq \int_0^{+\infty}\left(\int_0^t h(s) ds\right)^{\lambda-1} h(t) dt\leq \left(\int_0^{+\infty} h(t) dt\right)^\lambda.
\]
Choosing $t:=r^{N-ps}$, $h(t):=u(t^{1/(N-ps)})^p$,  $\lambda:=q/p$ yields
\[
\int_{\R^N} \frac{u^p}{|x|^{ps}}\, dx=\frac{N\omega_N}{N-ps}\int_0^{+\infty} h(t)\, dt\geq \frac{N\omega_N}{N-ps}\left(\int_0^{+\infty} u(t^{\frac{1}{N-ps}})^q\, t^{\frac{q}{p}-1} dt\right)^{\frac{p}{q}}.
\]
Since $t=r^{N-ps}$, from \eqref{scalingrelation} it follows
\[
\int_{\R^N} \frac{u^p}{|x|^{ps}}\, dx\geq \frac{N\omega_N}{N-ps}\left((N-ps)\int_0^{+\infty} \frac{u^q}{r^\alpha} r^{N-1}\, dr\right)^{\frac{p}{q}}= \left(\frac{N\omega_N}{N-ps}\right)^{-\frac{\alpha-ps}{N-\alpha}}\left(\int_{\R^N} \frac{u^q}{|x|^\alpha}\, dx\right)^{\frac{p}{q}}.
\]
Now, rearranging and inserting inside \eqref{HS} with $q:=p$ and $\alpha:=ps$ leads to the conclusion through P\'olya-Szeg\"o principle.
\end{proof}
Finally, we collect here the following two technical lemmas.
\begin{lemma}
\label{23}
Let $\eta\in C^\infty(\R^N)$ and let $\gamma>0$. Then
\begin{equation}\label{deta1}
\||D^s \eta|^\gamma\|_\infty\leq C(N, \gamma)\, \left(\frac{1}{1-s}+\frac{1}{s}\right)\,{\rm Lip}(\eta)^{\gamma s}\, \|\eta\|_\infty^{\gamma(1-s)}.
\end{equation}
If, moreover, ${\rm supp}(\eta)\subseteq B_R$ then for every $\theta>0$ there exists $C_\theta:=C(\theta, N, \gamma, s)$ such that
\begin{equation}\label{deta2}
|D^s\eta|^\gamma(x)\leq \frac{C_\theta\, R^N\, |\eta|_\infty^\gamma}{|x|^{N+\gamma s}},\qquad x\in B_{(1+\theta) R}^c.
\end{equation}
\end{lemma}
\begin{proof}
We may assume finite the right-hand side of \eqref{deta1}. Pick $\lambda>0$ and observe that
\[
\begin{split}
|D^s\eta|^\gamma(x)&\leq \left(\int_{B_{\lambda}(x)}\frac{{\rm Lip}(\eta)^\gamma\, |x-y|^\gamma}{|x-y|^{N+\gamma s}}\, dy+\int_{B_{\lambda}^c(x)}\frac{2^\gamma\, \|\eta\|_\infty^\gamma}{|x-y|^{N+\gamma s}}\, dy\right)\\
& \leq C\left( \frac{\lambda^{\gamma(1-s)}}{\gamma(1-s)}\, {\rm Lip}(\eta)^\gamma+\frac{2^\gamma\, \|\eta\|_\infty^\gamma}{\lambda^{\gamma s}\gamma s}\right),
\end{split}
\]
where $C=C(N)$. Optimizing in $\lambda>0$ this inequality directly yields \eqref{deta1}.\\
Let us next come to \eqref{deta2}. If $|x|\geq(1+\theta)R$ and  $|y|\leq R$ then
\[
|x|\leq |x-y|+|y|\leq |x-y|+R\leq |x-y|+\frac{|x|}{1+\theta}\quad \Rightarrow\quad |x|\leq \frac{1+\theta}{\theta}|x-y|.
\]
Since ${\rm supp}(\eta)\subseteq B_R$, one has
\[
|D^s\eta|^\gamma(x)=\int_{B_R}\frac{|\eta(y)|^\gamma}{|x-y|^{N+\gamma s}}\, dy\leq \left(\frac{1+\theta}{\theta}\right)^{N+\gamma s}\frac{\|\eta\|_\gamma^\gamma}{|x|^{N+\gamma s}},
\]
which easily entails \eqref{deta2}.
\end{proof}
Inspecting the previous proof one can also show that for $\gamma\leq 1$ it holds $C(N, \gamma)=C(N)/\gamma$.

\begin{lemma}\cite[Lemma A.2]{BP}\label{lemmag}
Suppose $g:\R\to \R$ is absolutely continuous continuous and non-decreasing, $p>1$, and
\begin{equation}
\label{defG}
G(t):=\int_0^tg'(\tau)^\frac{1}{p}\, d\tau,\quad t\in\R.
\end{equation}
 Then $[G(u)]_{s,p}^p\leq \langle (-\Delta_p)^s u, g(u)\rangle$ for every $u\in \dot W^{s,p}(\R^N)$.
\end{lemma}
\section{Concentration Compactness}
In this section we prepare details of the concentration-compactness scheme for problem \eqref{I}. Some arguments will closely follow \cite[Theorem 2.4]{Lions2}, but serious modifications, which we will explicitly outline below, are needed in order to deal with nonlocal interactions.
\begin{theorem}\label{exist}
Let $N>ps$, $q>p$, and $\alpha\in[0, ps[$ satisfy \eqref{scalingrelation}. Then \eqref{I} possesses a nonnegative radially decreasing minimizer.
\end{theorem}
\begin{proof}
It suffices to verify the assertion for $\lambda=1$. Define, provided $u\in \dot{W}^{s,p}(\R^N)$,
\[
\rho(u):=|D^s u|^p + |u|^{p^*}\in L^1(\R^N).
\]
If $\{u_n\}$ is a minimizing sequence of \eqref{I} (where $\lambda=1$) then, up to subsequences, 
\[
\lim_{n\to+\infty}\int_{\R^N}\rho(u_n)\, dx= L\geq I_1>0.
\]
We can choose a rescaling as well as a subsequence, still denoted by $\{u_n\}$, such that, setting
\[
 Q_n(t):=\sup_{y\in \R^N} \int_{B_t(y)}\rho(u_n)\, dx,\quad n\in\N,\; t\geq 0,
 \]
 one has
\[
Q_n(1)=I_1/2\quad\forall\, n\in\N, \quad \lim_{n\to+\infty}Q_n(t)=Q(t)\quad\forall\, t\geq 0.
\]
Therefore, vanishing cannot evidently occur in \cite[Lemma I.1]{Lions}. We will show that the same holds true for dichotomy; this is the point where nonlocal effects force a modification of the standard proof. Suppose on the contrary
\[
\lim_{t\to +\infty} Q(t)=a\in\  ]0, L[.
\]
Then for every $\eps>0$ there exist $\{y_n\}\subseteq\R^N$, $R_n\geq R_0>0$, $R_n\uparrow +\infty$ such that
\begin{equation}\label{dich}
\left|\int_{B_{R_0}(y_n)}\rho(u_n)\, dx-a\right|+\left|\int_{B_{R_n}^c(y_n)}\rho(u_n)\, dx-(L-a)\right|+\left|\int_{B_{R_n}(y_n)\setminus B_{R_0}(y_n)}\rho(u_n)\, dx\right|<\eps.
\end{equation}

The following elementary but useful inequalities quantify how the local behaviour of $|D^s u|$ controls the magnitude of $u$ nearby. Clearly, specific numbers are to some extent arbitrary choices.
 
\begin{lemma}
%\label{lemmanl}
Let $u\in \dot{W}^{s,p}(\R^N)$ and let $R>0$. Then
\begin{equation}\label{nl1}
\frac{1}{R^{ps}}\int_{B_R} |u|^p\, dx\leq C\int_{B_{3R}\setminus B_{2R}}|D^su|^p+\frac{1}{R^{ps}}|u|^p\, dx,
\end{equation}
\begin{equation}\label{nl2}
R^{N}\int_{B_{4R}^c} \frac{|u|^p}{|x|^{N+ps}}\, dx\leq C\int_{B_{3R}\setminus B_{2R}}|D^su|^p+\frac{1}{R^{ps}}|u|^p\, dx,
\end{equation}
with appropriate constant $C=C(N,p,s)$.
\end{lemma}
\begin{proof}
By density, we may assume $u\in C^\infty_c(\R^N)$ while via scaling one can put $R=1$. Let us first observe that $z\in B_3\setminus B_2$, $x\in B_1$ imply $1\leq |x-z|\leq 4$, whence
\[
\int_{B_1} |u|^p\, dx\leq 2^{p-1}\int_{B_1}|u(x)-u(z)|^p+ |u(z)|^p dx\leq C\left(\int_{B_1}\frac{|u(x)-u(z)|^p}{|x-z|^{N+ps}}\, dx+|u(z)|^p\right).
\]
After integration in $z\in B_3\setminus B_2$, this entails
\[
\begin{split}
\int_{B_1} |u|^p\, dx&\leq C\left(\int_{B_1\times (B_3\setminus B_2)}\frac{|u(x)-u(z)|^p}{|x-z|^{N+ps}}\, dx\, dz+\int_{B_3\setminus B_2}|u(z)|^p\, dz\right)\\
&\leq C\int_{B_3\setminus B_2}|D^su|^p+|u|^p\, dx,
\end{split}
\]
as desired. Similarly, since $|x-z|\leq |x|+|z|\leq 2|x|$ for all $x\in B_4^c$, $z\in B_3\setminus B_2$, one has 
\[
\int_{B_4^c} \frac{|u|^p}{|x|^{N+ps}}\, dx\leq C\left(\int_{B_4^c}\frac{|u(x)-u(z)|^p}{|x-z|^{N+ps}}\, dx+|u(z)|^p\right),
\]
which integrated in $z\in B_3\setminus B_2$ provides 
\[
\int_{B_{4}^c} \frac{|u|^p}{|x|^{N+ps}}\, dx\leq C\int_{B_{3}\setminus B_{2}}|D^su|^p+|u|^p\, dx,
\]
and \eqref{nl2} follows.
\end{proof}
\begin{lemma}\label{lemmaint}
Suppose $\eta\in C^\infty(\R^N)$ fulfills $0\leq \eta\leq 1$, $\eta\lfloor_{B_4}=1$, $\eta\lfloor_{ B_5^c}=0$, and define
$\eta_R(x):=\eta(x/R)$, $R>0$. Then there exists $C_\eta:=C(N, p, s, {\rm Lip}(\eta))$ such that for every $u\in \dot{W}^{s,p}(\R^N)$ one has
\begin{equation}\label{st1}
\left|\int_{B_R}|D^su|^p\, dx-\int_{\R^N}|D^s(\eta_R\,  u)|^p\, dx\right|\leq C_\eta\int_{B_{20R}\setminus B_{R}}|D^s u|^p\, dx+\frac{C_{\eta}}{R^{ps}}\int_{B_{20R}\setminus B_{R}}|u|^p\, dx.
\end{equation}
Suppose $\xi\in C^\infty(\R^N)$ fulfills $0\leq \xi\leq 1$, $\xi\lfloor_{B_{1/5}}=0$, $\xi\lfloor_{B_{1/4}^c}=1$, and define
$\xi_R(x):=\xi(x/R)$, $R>0$. Then there exists $C_\xi:=C(N, p, s, {\rm Lip}(\xi))$ such that for every $u\in \dot{W}^{s,p}(\R^N)$ one has
\begin{equation}\label{st2}
\left|\int_{B_{R}^c}|D^su|^p\, dx-\int_{\R^N}|D^s(\xi_R\,  u)|^p\, dx\right|\leq C_\xi\int_{B_{R}\setminus B_{R/20}}|D^s u|^p\, dx+\frac{C_\xi}{R^{ps}}\int_{B_{R}\setminus B_{R/20}}|u|^p\, dx.
\end{equation}
\end{lemma}
\begin{proof}
We start by proving \eqref{st1}. Scale to $R=1$ and observe that
\begin{equation}\label{diff}
\begin{split}
\int_{B_1}|D^s u|^p\, dx-&\int_{\R^N}|D^s(\eta\,  u)|^p\, dx=-\int_{B_1^c\times \R^N}\frac{|\eta(x)\, u(x)-\eta(y)\, u(y)|^p}{|x-y|^{N+ps}}\, dx\, dy\\
&\quad +\int_{B_1\times B_{4}^c}\frac{|u(x)-u(y)|^p-|u(x)\, \eta(x)-\eta(y)\,  u(y)|^p}{|x-y|^{N+ps}}\, dx\, dy.
\end{split}
\end{equation}
Now,
\[
\begin{split}
\int_{B_1\times B_{4}^c}&\frac{|u(x)-u(y)|^p}{|x-y|^{N+ps}}\, dx\, dy\leq 
2^{p-1}\int_{B_1\times B_{4}^c}\frac{|u(x)|^p+|u(y)|^p}{|x-y|^{N+ps}}\, dx\, dy\\
&\leq C\left(\int_{B_1}|u|^p\, dx+\int_{B_4^c}\frac{|u|^p}{|y|^{N+ps}}\, dy\right)
\leq C \int_{B_3\setminus B_2}|D^su|^p+|u|^p\, dx.
\end{split}
\]
A similar inequality is true for the term involving $\eta$. Let us next estimate the other integral that appears in \eqref{diff}. Evidently,
\begin{equation}\label{lieb}
|\eta(x)\, u(x)-\eta(y)\, u(y)|^p\leq 2^{p-1}\left(|\eta(x)|^p\, |u(x)-u(y)|^p+|\eta(x)-\eta(y)|^p\, |u(y)|^p\right)
\end{equation}
and 
\[
\int_{B_1^c\times \R^N}\frac{|\eta(x)|^p\, |u(x)- u(y)|^p}{|x-y|^{N+ps}}\, dx\, dy\leq C\int_{B_5\setminus B_1}|D^s u|^p\, dx.
\]
Exploiting \eqref{deta1} on $B_6$ and \eqref{deta2} on $B_6^c$ (recall that ${\rm supp}(\eta)\subseteq B_5$), we get
\begin{equation*}%\label{temp0}
\begin{split}
\int_{B_1^c\times \R^N}&\frac{|\eta(x)-\eta(y)|^p\,  |u(y)|^p}{|x-y|^{N+ps}}\, dx\, dy\leq \int_{\R^N}|D^s\eta|^p\, |u|^p\, dy\\
&\leq \int_{B_6}|D^s\eta|^p\, |u|^p\, dy+\int_{B_6^c}|D^s\eta|^p\, |u|^p\, dy\\
&\leq C\, (1+{\rm Lip}(\eta)^{ps})\left(\int_{B_6}|u|^p\, dy+\int_{B_6^c}\frac{|u|^p}{|y|^{N+ps}}\, dy\right).
\end{split}
\end{equation*}
Thanks to \eqref{nl1}--\eqref{nl2}, this entails
\begin{equation*}
\begin{split}
\int_{B_1^c\times \R^N}&\frac{|\eta(x)-\eta(y)|^p\, |u(y)|^p}{|x-y|^{N+ps}}\, dx\, dy\\
&\leq C_\eta\int_{B_{18}\setminus B_{12}}|D^s u|^p+|u|^p \, dx+ C_\eta\int_{B_{9/2}\setminus B_{3}}|D^s u|^p+|u|^p \, dx.
\end{split}
\end{equation*}
Gathering together the estimates above we achieve \eqref{st1}.

The proof of \eqref{st2} is entirely analogous but, for the sake of completeness, we sketch it. Since 
\[
\begin{split}
&\int_{B_{1}^c}|D^su|^p\, dx-\int_{\R^N}|D^s(\xi\, u)|^p\, dx=\\
&\int_{B_1^c\times B_{1/4}}\frac{|u(x)-u(y)|^p-|\xi(x)\, u(x)-\xi(y)\, u(y)|^p}{|x-y|^{N+ps}}\, dx\, dy-\int_{B_1}|D^s(\xi\, u)|^p\, dx,
\end{split}
\]
let us estimate the two terms separately. Concerning the first one,
\[
\begin{split}
\left|\int_{B_1^c\times B_{1/4}}\frac{|u(x)-u(y)|^p-|u(x)-\xi(y)\, u(y)|^p}{|x-y|^{N+ps}}\, dx\, dy\right|
&\leq 2^p\int_{B_1^c\times B_{1/4}}\frac{|u(x)|^p+|u(y)|^p}{|x-y|^{N+ps}}\, dx\, dy\\
&\leq C\int_{B_1^c}\frac{|u|^p}{|x|^{N+ps}}\, dx +C\int_{B_{1/4}}|u(y)|^p\, dy,
\end{split}
\]
which, via a suitable rescaling of \eqref{nl1}--\eqref{nl2}, provides
\[
\left|\int_{B_1^c\times B_{1/4}}\frac{|u(x)-u(y)|^p-|u(x)-\xi(y)\, u(y)|^p}{|x-y|^{N+ps}}\, dx\, dy\right|\leq C\int_{B_{3/4}\setminus B_{1/2}}|D^s u|^p+|u|^p\, dx.
\]
For the second one, we proceed as in \eqref{lieb} and obtain 
\[
\begin{split}
\int_{B_1}|D^s(\xi \, u)|^p\, dx&\leq C\left(\int_{B_1}|\xi|^p\, |D^s u|^p\, dx+\int_{\R^N}|u|^p\, |D^s\xi|^p\, dx\right)\\
&\leq C\left(\int_{B_1\setminus B_{1/5}}|D^s u|^p\, dx+\int_{\R^N}|u|^p\, |D^s\xi|^p\, dx\right).
\end{split}
\]
Since $|D^s\xi|^p=|D^s(1-\xi)|^p$, due to \eqref{deta1}--\eqref{deta2} one has
\[
\begin{split}
\int_{\R^N}|u|^p\, |D^s\xi|^p\, dx&=\int_{B_{1/6}}|u|^p\, |D^s(1-\xi)|^p\, dx+\int_{B_{1/6}^c}|u|^p\, |D^s(1-\xi)|^p\, dx\\
&\leq C_\xi\left(\int_{B_{1/6}}|u|^p\, dx +\int_{B_{1/6}^c}\frac{|u|^p}{|x|^{N+ps}}\, dx\right).
\end{split}
\]
Through a suitable rescaling of \eqref{nl1}--\eqref{nl2}, it yields
\[
\int_{\R^N}|u|^p\, |D^s\xi|^p\, dy\leq 
C_\xi\int_{B_{1/2}\setminus B_{1/3}}|D^s u|^p+|u|^p\, dx+C_\xi\int_{B_{1/8}\setminus B_{1/12}}|D^s u|^p+|u|^p\, dx.
\]
Gathering together the above estimates we arrive at \eqref{st2}.
\end{proof}
Pick $\eta, \xi$ as in Lemma \ref{lemmaint} and define 
\[
u_n^1(x):=\eta\big(\frac{x-y_n}{R_0}\big)\, u_n(x),\quad u_n^2(x)=\xi\big(\frac{x-y_n}{R_n}\big)\, u_n(x).
\]
If $R_n>400R_0$ then, by \eqref{st1}--\eqref{st2}
\[
\begin{split}
\left|\int_{B_{R_0}(y_n)}\right. & \left.|D^s u_n|^p\, dx-\int_{\R^N}|D^s u_n^1|^p\, dx\right|+\left|\int_{B_{R_n}^c(y_n)}|D^s u_n|^p\, dx-\int_{\R^N}|D^s u_n^2|^p\, dx\right|\\
&\leq C_\eta\left(\int_{B_{20R_0}(y_n)\setminus B_{R_0}(y_n)}|D^s u_n|^p\,dx +\frac{1}{R_0^{ps}}\int_{B_{20R_0}(y_n)\setminus B_{R_0}(y_n)} |u_n|^p\, dx\right)\\
&\quad +C_\xi\left(\int_{B_{R_n}(y_n)\setminus B_{\frac{R_n}{20}}(y_n)}|D^s u_n|^p\,dx +\frac{1}{R_n^{ps}}\int_{B_{R_n}(y_n)\setminus B_{\frac{R_n}{20}}(y_n)} |u_n|^p\, dx\right)\\
&\leq C_\eta\int_{B_{20R_0}(y_n)\setminus B_{R_0}(y_n)}|D^s u_n|^p\,dx +C'_\eta\left(\int_{B_{20R_0}(y_n)\setminus B_{R_0}(y_n)} |u_n|^{p^*}\, dx\right)^{\frac{p}{p^*}}\\
&\quad +C_\xi\int_{B_{R_n}(y_n)\setminus B_{\frac{R_n}{20}}(y_n)}|D^s u_n|^p\,dx +C'_\xi\left(\int_{B_{R_n}(y_n)\setminus B_{\frac{R_n}{20}}(y_n)} |u_n|^{p^*}\, dx\right)^{\frac{p}{p^*}}\\
&\leq C_1\int_{B_{R_n}(y_n)\setminus B_{R_0}(y_n)} \rho(u_n)\, dx+C_2\left(\int_{B_{R_n}(y_n)\setminus B_{R_0}(y_n)} \rho(u_n)\, dx\right)^{\frac{p}{p^*}}\leq C_3(\eps+\eps^{\frac{p}{p^*}}).
\end{split}
\]
Therefore, due to \eqref{dich},
\begin{equation}\label{split0}
\left|[u_n]_{s,p}^p-[u_n^1]_{s,p}^p-[u_n^2]_{s,p}^p\right|\le C\, (\eps+\eps^{\frac{p}{p^*}}).
\end{equation}
Concerning $L^{p^*}$-norms, we readily have 
\[
\left|\int_{B_{R_0}(y_n)}|u_n|^{p^*}\, dx-\|u_n^1\|_{p^*}^{p^*}\right|+\left|\int_{B_{R_n}^c(y_n)}|u_n|^{p^*}\, dx-\|u_n^2\|_{p^*}^{p^*}\right|\le C\int_{B_{R_n}(y_n)\setminus B_{R_0}(y_n)}|u_n|^{p^*}\, dx,
\]
whence, in view of \eqref{dich} again, 
\begin{equation}
\label{split2}
\left|\int_{\R^N}\rho(u_n^1)\, dx-a\right|+\left|\int_{\R^N}\rho(u_n^2)\, dx-(L-a)\right|< C\, (\eps+\eps^{\frac{p}{p^*}}),\quad a\in \ ]0, L[.
\end{equation}
To conclude, suppose that
\[
\lim_{n\to+\infty}\int_{\R^N} \frac{|u_n^1|^q}{|x|^\alpha}\, dx=\lambda_1,\quad \lim_{n\to+\infty}\int_{\R^N}\frac{|u_n^2|^q}{|x|^\alpha}\, dx=\lambda_2 
\]
(where a subsequence is considered when necessary) with appropriate $\lambda_1,\lambda_2\in [0, 1]$ depending on $\eps$, i.e., $\lambda_1:=\lambda_1(\eps)$, $\lambda_2:=\lambda_2(\eps)$, and put 
\[
\eta_n(x):=\eta(\frac{x-y_n}{R_0}),\quad \xi_n(x):=\xi(\frac{x-y_n}{R_n}),\quad
\theta_n(x):=(1-\eta_n(x)^q-\xi_n(x)^q)^{1/q}. 
\]
Clearly, $0\le \theta_n\le \chi_{B_{R_n}(y_n)\setminus B_{R_0}(y_n)}$ because $R_n\geq 400 R_0$. By H\"older's  and Hardy's inequalities, besides \eqref{scalingrelation}, one has
\[
\begin{split}
\int_{\R^N}(|u_n|^q&-|u_n^1|^q-|u_n^2|^q)\frac{dx}{|x|^\alpha}=\int_{\R^N}|\theta_n\,  u_n|^q\, \frac{dx}{|x|^\alpha}\le \int_{\R^N}\frac{|u_n|^{\frac{\alpha}{s}}}{|x|^\alpha}|\theta_n\, u_n|^{q-\frac{\alpha}{s}}\, dx\\
&\le \| u_n\|_{ps, p}^{\frac{\alpha}{s}}\, \|\theta_n \,u_n \|_{p^*}^{p^*(1-\frac{\alpha}{ps})}\le C\left(\int_{B_{R_n}(y_n)\setminus B_{R_0}(y_n)}\rho(u_n)\, dx\right)^{1-\frac{\alpha}{ps}}\le C\, \eps^{1-\frac{\alpha}{ps}}.
\end{split}
\]
Thus,
\begin{equation}\label{alphabeta}
|\lambda_1+\lambda_2-1|= \lim_{n\to+\infty}\left|\int_{\R^N}(|u_n|^q-|u_n^1|^q-|u_n^2|^q)\frac{dx}{|x|^\alpha}\right|\leq C\, \eps^{1-\frac{\alpha}{s}}.
\end{equation}
The remaining proof of tightness now follows verbatim from \cite[Section I.2, Step 1]{Lions1}. Nevertheless, we briefly sketch it here. Via \eqref{split2} and Sobolev's inequality, one can find $b>0$, $\eps_0>0$ such that
\[
[u^1_n]_{s,p}^p\geq b,\qquad [u^2_n]_{s,p}^p \geq b
\]
for all $\eps<\eps_0$. Consequently, both numbers $\lambda_J$ are bounded away from $0$ or $1$ provided $\eps$ is sufficiently small. Indeed, using \eqref{split0}, the above inequalities, and \eqref{Ilambda} yields
\[
I_1=\lim_{n\to+\infty}[u_n]^p_{s,p}\geq \lim_{n\to+\infty}\left([u_n^1]^p_{s,p}+[u^2_n]_{s,p}^p\right)-O(\eps)\geq b+I_{\lambda_2}-O(\eps)=b+\lambda_2^{\frac{p}{q}} I_1-O(\eps).
\]
This entails a bound from above to $\lambda_2$, as long as $O(\eps)<b/2$, and also a bound from below for $\lambda_1$, thanks to \eqref{alphabeta}. A similar reasoning furnishes a bound from above for $\lambda_1$. Hence,
\[
I_1\geq \lim_{n\to+\infty}[u_n^1]^p_{s,p}+[u^2_n]_{s,p}^p\geq I_{\lambda_1}+I_{\lambda_2}-O(\eps).
\]
Now, pick $\eps:=\eps_k\to 0$. Up to subsequences, we evidently have $\lambda_1(\eps_k)\to\bar\lambda\in \ ]0, 1[$ and, by \eqref{alphabeta}, $\lambda_2(\eps_k)\to 1-\bar\lambda\in \ ]0, 1[$. So, due to \eqref{Ilambda}, 
$$I_1\geq\bar\lambda^{\frac{p}{q}}I_1+(1-\bar\lambda)^{\frac{p}{q}}I_1,$$
which is impossible whenever $q>p$. 

Finally, Conclusion (i) of \cite[Lemma I.1]{Lions} produces a sequence $\{y_n\}\subseteq\R^N$ with the following property:  
\[
\mbox{For every $\eps>0$ there exists $R>0$ such that}\;\int_{B_R^c(y_n)}\rho(u_n)\, dx<\eps,\;\; n\in\N.
\]
Let us next show that $\{y_n\}$ is bounded. To do this, pick $\eta$ as in Lemma \ref{lemmaint} and define $\eta_n(x):=\eta((x-y_n)/R)$, $u_n^1(x):=u_n(x)\, \eta_n(x)$. Inequality \eqref{Liebn} provides
\[
[u_n-u_n^1]_{s,p}^p\leq 2^{p-1}\left(\int_{\R^N}|1-\eta_n|^p\, |D^s u_n|^p\, dx+\int_{\R^N}|u_n|^p\, |D^s\eta_n|^p\, dx\right),
\]
while, by construction,
\[
\begin{split}
\int_{\R^N}|1-\eta_n|^p\, |D^s u_n|^p\, dx&\leq\int_{B^c_{4R}(y_n)}|D^s u_n|^p\, dx\leq\int_{B_R^c(y_n)}\rho(u_n)\, dx
<\eps,\\
\int_{\R^N}|u_n|^p\, |D^s\eta_n|^p\, dx&=\int_{B_{6R}(y_n)}|u_n|^p\, |D^s\eta_n|^p\, dx+
\int_{B^c_{6R}(y_n)}|u_n|^p\, |D^s\eta_n|^p\, dx.
\end{split}
\]
Since ${\rm Lip}(\eta_n)={\rm Lip}(\eta)/R$, through \eqref{deta1}, \eqref{nl1} (rescaled), and H\"older's inequality, we obtain
\[
\begin{split}
\int_{B_{6R}(y_n)} |u_n|^p\, |D^s\eta_n|^p\, dx&\leq \frac{C_1}{R^{ps}}\int_{B_{6R}(y_n)} |u_n|^p\, dx\leq C_2 \int_{B_{18R}(y_n)\setminus B_{12R}(y_n)}|D^s u_n|^p+\frac{|u_n|^p}{R^{ps}}\, dx\\
&\leq C_3\left[\int_{B_R^c(y_n)}\rho(u_n)\, dx+\left(\int_{B_R^c(y_n)}\rho(u_n)\, dx\right)^{\frac{p}{p^*}}\right]
\leq C_3(\eps+\eps^{\frac{p}{p^*}}).
\end{split}
\]
Analogously, on account of \eqref{deta2} and \eqref{nl2} (rescaled), one has
\[
\begin{split}
\int_{B_{6R}^c(y_n)} |u_n|^p\, |D^s\eta_n|^p\, dx&\leq C_4 \, R^N\int_{B_{6R}^c(y_n)}\frac{|u_n|^p}{|x|^{N+ps}}\, dx\leq
C_5\int_{B_{\frac{9}{2}R}\setminus B_{3R}}|D^s u_n|^p+\frac{|u_n|^p}{R^{ps}}\, dx\\
&\leq C_6\left[\int_{B_R^c(y_n)}\rho(u_n)\, dx+\left(\int_{B_R^c(y_n)}\rho(u_n)\, dx\right)^{\frac{p}{p^*}}\right]
\leq C_6(\eps+\eps^{\frac{p}{p^*}}).
\end{split}
\]
Gathering together the above inequalities produces 
\[
[u_n-u_n^1]_{s,p}^p\leq C\, (\eps+\eps ^{\frac{p}{p^*}}),
\]
and, to see that $\{y_n\}$ is bounded, we  proceed exactly as in \cite[p. 64]{Lions2}.

Finally, the compactness of $\{u_n\}$ stems from the Second Concentration-Compactness Lem\-ma as performed in \cite[Theorem 2.5]{MS}. It suffices to substitute $\|u\|_{p^*}$ with $\|u\|_{\alpha, q}$ in the proof.
\end{proof}
\section{Regularity estimates}

Recall that the weak-$L^q$ quasi-norm of a measurable function $u:\R^N\to \R$ is defined by setting
\[
\|u\|_{L^{q, \infty}}:=\sup_{k>0}k|\{|u(x)|>k\}|^{1/q}.
\]
While in the next lemma we consider arbitrary open $\Omega\subseteq \R^N$, we will be mainly interested in the case $\Omega=\R^N$.
\begin{theorem}[Summability estimates]\label{Rlemma}
Let $N>ps$, let $\Omega\subseteq\R^N$ be nonempty open, and let $f\in L^r(\R^N)$ for some $r\geq 1$. Suppose $u\in W^{s,p}_0(\Omega)$ weakly solves $(-\Delta_p)^su= f$ in $\Omega$, in the sense of \eqref{locweak}. Then there exists a constant $C>0$ such that
\begin{align}
\label{stimalr4}
\|u\|_{L^{\frac{p^*}{p'}, \infty}}&\leq C\, \|f\|_1^{\frac{1}{p-1}} & \quad & \text{if}\quad  r=1,\\ 
\label{stimalr1}
\|u\|_t&\leq C\, \|f\|_r^{\frac{1}{p-1}}&\quad &\text{if }\quad 1<r<\frac{N}{ps},\; t=\frac{N(p-1)r}{N-psr},\\
\label{stimalr2}
\|u\|_t&\leq C\, \|f\|_{N/ps}^{\frac{t-p^*}{t(p-1)}}\, \|u\|_{p^*}^{\frac{p^*}{t}}&\quad &\text{if}\quad r=\frac{N}{ps},\; t\geq p^*,\\
\label{stimalr3}
\|u\|_{\infty}&\leq C\, \|f\|_r^{\frac{r'}{p^*-r'}}\, \|u\|_{p^*}^{\frac{p^*-pr'}{p^*-r'}}&\quad &\text{if}\quad \frac{N}{ps}<r\leq +\infty.
\end{align}
The constant $C$ depends only on $N, p, s, r$ and possibly $t$ in the case $r=\frac{N}{ps}$.
\end{theorem}
\begin{proof}
Given $k>\eps>0$, $\beta\geq 1$ we define
$$t_{k, \eps}:=\min\{k, (t-\eps)_+\},\quad g_\beta(t):=(t_{k, \eps})^\beta\quad\forall\, t\in\R.$$ 
Clearly, $g$ is non-decreasing, Lipschitz continuous, and
\begin{equation}\label{constG}
G_\beta(t)=\frac{\beta^{1/p} p}{\beta+p-1} (t_{k, \eps})^{\frac{\beta+p-1}{p}},
\end{equation}
with $G_\beta$ as in \eqref{defG}. Moreover, $g_\beta\circ u\in W^{s,p}_0(\Omega)$ turns out to be a suitable test function, because it is bounded and has a finite measure support. Thus, using Lemma \ref{lemmag},  Sobolev inequality on the left, and H\"older inequality on the right, yields
\begin{equation}\label{moser}
C\left\| u_{k, \eps}^{\frac{\beta+p-1}{p}}\right\|_{p^*}^p\leq [G_\beta(u)]_{s, p}^p\leq \langle(-\Delta_p)^s u, g_\beta\circ u\rangle=\int_\Omega f \, g_\beta\circ u\, dx\leq\|f\|_r\, \|u_{k, \eps}^\beta\|_{r'}
\end{equation}
for some $C=C(N, p, s, \beta)>0$ and any $r\geq 1$.

{\em Case 1: $r=1$ (whence $r'=\infty$).}\\
Pick $\beta:=1$ in \eqref{moser}. By the Tchebychev inequality one has
$$k^p\, |\{|u|\geq k\}^{\frac{p}{p^*}}\leq \|u_{k, \eps}\|_{p^*}^p\leq C\, \|f\|_1\, \|u_{k, \eps}\|_\infty\leq C\, \|f\|_1 \, k,$$
which easily entails \eqref{stimalr4} once $\eps\to 0^+$ and the supremum over $k>0$ is taken. 

{\em Case 2: $1<r<\frac{N}{ps}$ and $r'\leq p^*$}. \\
These inequalities force
\begin{equation}\label{betazero}
\beta_0(r):=\frac{(p-1)p^*}{pr'-p^*}\geq 1
\end{equation}
as well as
\begin{equation}\label{condbeta}
\frac{\beta_0+p-1}{p}p^*=\beta_0 r'=\frac{N(p-1)r}{N-psr}.
\end{equation}
If $\beta:=\beta_0$  then \eqref{moser} becomes \eqref{stimalr1} with $u:=u_{k, \eps}$. Letting $k\to +\infty$, $\eps\to 0^+$ we achieve the conclusion.

{\em - Case 3: $1<r<\frac{N}{ps}$ and $r'> p^*$}. \\
In this case, $0<\beta_0<1$, with $\beta_0$ given by \eqref{betazero}, and $g$ is no longer Lipschitz continuous.
Define, provided $k>\eps>0$,
$$\tilde g(t):=\min\{k^{\beta_0}, \max\{t, \eps\}^{\beta_0}-\eps^{\beta_0}\} \quad \forall\, t\in\R^+_0,\quad
\tilde g (t):=-\tilde g(-t)\quad\forall\, t\in\R^-.$$
The inequality
\[
\left(\frac{\beta_0^{1/p}p}{\beta_0+p-1}\right)^{p^*} |\tilde g(t)|^{r'}\leq |\tilde G(t)|^{p^*}
\]
is reduced to
\[
\frac{(\tau^q-1)^{1/q}}{\tau-1}\geq 1,\qquad q:=\frac{r'}{p}>1,\qquad  \tau:=(t/\eps)^{\beta_0}\geq 1,
 \]
which can be verified via elementary considerations. Observe also that $\tilde g$ is Lipschitz continuous. So, $\tilde g\circ u\in W^{s,p}_0(\Omega)$ turns out to be a suitable test function, because it is bounded and has finite measure support. On account of \eqref{condbeta}, the same argument employed for proving \eqref{moser} produces here
$$C\, \|\tilde g\circ u\|_{r'}^{\frac{pr'}{p^*}}\leq \|f\|_{r}\, \|\tilde g\circ u\|_{r'}.$$
As before, this entails \eqref{stimalr1}.

{\em Case 4: $r\geq \frac{N}{ps}$}.\\
Without loss of generality, we may suppose $\|u\|_{p^*}=\|f\|_r=1$. Indeed, if  $v^{\lambda, \mu}(x):=\lambda v (\mu x)$ for every $\lambda, \mu>0$ and measurable $v:\R^N\to\R$, then
\[
(-\Delta_p)^s u^{\lambda, \mu}=\lambda^{p-1}\mu^{ps} f(\mu x)=f^{\lambda^{p-1} \mu^{ps}, \mu}.
\]
Since there obviously exist $\bar\lambda, \bar\mu>0$ such that $\|u^{\bar\lambda, \bar\mu}\|_{p^*}=\|f^{\bar\lambda^{p-1}\bar\mu^{ps}, \bar\mu}\|_r=1$, showing \eqref{stimalr2}--\eqref{stimalr3} for $u^{\bar\lambda, \bar\mu}$ actually gives the general case by scaling and homogeneity.  Define 
\begin{equation*}
%\label{exponents}
\tilde\beta_0:=p^*,\quad  \tilde\beta_{n+1}:=p^*\frac{\frac{\tilde\beta_{n}}{r'}+p-1}{p},
\end{equation*}
and test the equation $(-\Delta_p)^su=f$ with $(u_{k, \eps})^{\beta_n}$, where
$$\beta_n:=\tilde\beta_n/r'\geq\tilde\beta_0/r'\geq p>1.$$
Then \eqref{moser} reads
\begin{equation}\label{moser2}
C_{n+1}\left\|u_{k, \eps}\right\|_{\tilde\beta_{n+1}}^{\tilde\beta_{n+1}\frac{p}{p^*}}\leq \left\|u_{k, \eps}\right\|_{\tilde\beta_n}^{\frac{\tilde\beta_n}{r'}},\quad C_1\|u_{k, \eps}\|_{\tilde\beta_1}^{\tilde\beta_{1}\frac{p}{p^*}}\leq 1
\end{equation}
because $\|u\|_{p^*}=\|f\|_r=1$. Now, from $r\geq \frac{N}{ps}$ it follows $\tilde\beta_n\to +\infty$ as $n\to +\infty$. So, if $t\geq p^*$ then $\tilde\beta_{n}\leq t\leq \tilde\beta_{n+1}$ for some $n\in\N$. By interpolation one has
$$\|u_{k, \eps}\|_t\leq \|u_{k, h}\|_{\tilde\beta_n}^\theta\|u_{k, \eps}\|_{\tilde\beta_{n+1}}^{1-\theta}\leq 
C'_n \|u_{k, \eps}\|_{\tilde\beta_{n}}^{a_{n}}\leq C'_{n-1} \|u_{k, \eps}\|_{\tilde\beta_{n-1}}^{a_{n-1}}\leq  \dots\leq 
C'_0\|u_{k, \eps}\|_{p^*}^{a_0}= C_0'(t)$$
with appropriate $a_n, C'_n>0$. Letting $k\to +\infty$ and $\eps\to 0^+$ yields \eqref{stimalr2} after scaling back. Finally, suppose $r>\frac{N}{ps}$. Through \eqref{constG} we achieve $C_n\geq C/\tilde\beta_n^{p-1}$ for any sufficiently large $n$.
This polynomial decay ensures that \eqref{moser2} can be iterated {\em ad infinitum} provided $\{\tilde\beta_n\}$ grows geometrically, which holds true being $r>\frac{N}{ps}$. One thus has
\[
\|u_{k, \eps}\|_\infty=\lim_{n\to +\infty}\|u_{k, \eps}\|_{\tilde\beta_{n+1}}\leq \lim_{n\to +\infty}C_n^{-\frac{p^*}{p\tilde\beta_{n+1}}}\|u_{k, \eps}\|_{\tilde{\beta}_n}^{\frac{p^*}{pr'} \frac{\tilde\beta_n}{\tilde\beta_{n+1}}}\leq \dots \leq C(n, N, p, s), 
\]
and the proof of \eqref{stimalr3} goes on as before.
\end{proof}

The next corollary shows that the (lower) summability threshold at which $|(-\Delta_p)^s u|$ exhibits a better decay rate than the natural one is $r=(p^*)'$.
\begin{corollary}
\label{cordecay}
Let $N>ps$ and let $u\in \dot{W}^{s,p}(\R^N)$ be a radial, radially decreasing weak solution of $(-\Delta_p)^s u=f$ in $\R^N$,  where $f\in L^r(\R^N)$ for some $1\leq r\leq \frac{p^*}{p^*-1}$. Then, for a suitable $C=C(N, p, s)$ it holds
\begin{equation}\label{decay}
|u(R)|\leq \frac{C\, \|f\|_r^{\frac{1}{p-1}}}{R^{\frac{N-psr}{(p-1)r}}}\quad \forall\, R>0.
\end{equation}
\end{corollary}
\begin{proof}
The conclusion directly follows from \eqref{stimalr4}, \eqref{stimalr1}, and the decay estimates for radially decreasing functions in Lorentz spaces established in \cite[Lemma 2.9]{BMS}.  It suffices to observe that $N>ps$ forces $\frac{p^*}{p^*-1}<\frac{N}{ps}$ and that $r\leq \frac{p^*}{p^*-1}$ means $p^*\geq \frac{N(p-1)r}{N-psr}$.
\end{proof}

Notice that, if $r\geq \frac{p^*}{p^*-1}$, then the natural summability $u\in L^{p^*}(\R^N)$ provides a faster decay rate for radially decreasing functions than the one deduced from \eqref{stimalr1}--\eqref{stimalr3}, namely
$$|u(R)|\leq C\, \|u\|_{p^*}\, R^{-\frac{N}{p^*}},\quad R>0.$$

The following lemma represents a higher regularity estimate in Besov spaces. \begin{lemma}[Regularity estimate]\label{regest}
Let $p,r,t>1$ and $\theta\in \ ]0, 1]$  be such that
\begin{equation}\label{deftheta}
\frac{\theta}{p}+\frac{1-\theta}{t}=\frac{1}{r'}.
\end{equation}
Suppose $u\in L^t(\R^N) \cap \dot W^{s,p}(\R^N)$ and $f\in L^r(\R^N)\cap \dot{W}^{-s, p'}(\R^N)$ satisfy $(-\Delta_p)^s u=f$ weakly in $\R^N$, as per \eqref{locweak}. Then
\begin{align}\label{Bp>2}
\sup_{|h|>0}\left\|\frac{\delta^2_hu}{|h|^{\frac{sp}{p-\theta}}}\right\|_p&\leq C\, \|f\|_{r}^{\frac{1}{p-\theta}}\, \|u\|_t^{\frac{1-\theta}{p-\theta}}&&\text{if $p\geq 2$},\\
%
%\phantom{}\nonumber\\
\label{Bp<2}
\sup_{|h|>0}\left\|\frac{\delta^2_hu}{|h|^{\frac{2s}{2-\theta}}}\right\|_p&\leq C\, \|f\|_r^{\frac{1}{2-\theta}}\, \|u\|_t^{\frac{1-\theta}{2-\theta}}\, [u]_{s,p}^{\frac{2-p}{2-\theta}}& &\text{if $1<p<2$,}
\end{align}
with appropriate constant $C:=C(N, p, s, r, t)>0$.
\end{lemma}
\begin{proof}
Pick $h\in \R^N\setminus\{0\}$. By translation invariance one has
\[
\langle (-\Delta_p)^s u_h, \varphi\rangle =\int_{\R^N} f_h\, \varphi\, dx,\qquad 
\langle (-\Delta_p)^s u, \varphi\rangle =\int_{\R^N} f\, \varphi\, dx,
\]
which entails
\begin{equation}\label{eqdiff}
 \langle (-\Delta_p)^s u_h-(-\Delta_p)^s u, \varphi\rangle =\int_{\R^N}\varphi\, \delta_h f\, dx=\int_{\R^N} f\, \delta_{-h}\varphi\, dx.
\end{equation}
Observe next that $\delta^2_h u$ turns out to be a viable test function for a.e. $h\neq 0$. Indeed, from
$$[u]_{s,p}^p=\int_{\R^N}\|\delta_h u\|_p^p\, \frac{dh}{|h|^{N+ps}}<+\infty$$
we evidently infer $\|\delta_h u\|_p<+\infty$ for almost every $h$, even when neither $u$ nor $u_h$ lie in $L^p(\R^N)$. The continuity of $L^p$-norm with respect to translation yields $\delta_h u\in L^p(\R^N)$, whence $\delta^2_hu\in L^p(\R^N)$, because $\|\delta^2_h u\|_p\leq2\, \|\delta_h u\|_p$. Exploiting \eqref{deftheta} (with $\theta:=1$ if $p=r'=t$), H\"older's inequality and the inequality $\|\delta^2_h u\|_t\leq 4\, \|u\|_t$, easily provides
\begin{equation}\label{fd2u}
\|f\, \delta^2_h u\|_1\leq \|f\|_r\, \|\delta^2_h u\|_{r'}\leq \|f\|_r\, \|\delta^2_h u\|^\theta_p\, \|\delta^2_h u\|_t^{1-\theta}\leq 4\,  \|f\|_r\, \|\delta^2_h u\|^\theta_p\, \| u\|_t^{1-\theta}.
\end{equation} 
Hence, $f\, \delta_h^2 u\in L^1(\R^N)$, as desired, for a.e. $h\neq 0$. Since we will take the essential supremum in $h$, we can assume that this holds for any $h\neq 0$. We can thus set $\varphi:=\delta_hu$ in \eqref{eqdiff}, whose left-hand side becomes
\begin{equation}\label{lhs}
\begin{split}
&\langle (-\Delta_p)^s u_h-(-\Delta_p)^s u, \delta_h u\rangle=\\
&\int_{\R^{2N}}\frac{\left((u_h(x)-u_h(y))^{p-1}-(u(x)-u(y))^{p-1}\right)\left((u_h(x)-u_h(y))-(u(x)-u(y))\right)}{|x-y|^{N+ps}}
\,dx \,dy.
\end{split}
\end{equation}
Now, the proof naturally splits into two cases.

{\em Case 1: $p\geq 2$}.\\
The known inequality
\begin{equation}\label{dp1}
(a^{p-1}-b^{p-1})(a-b)\geq c_{p}|a-b|^p\quad \forall\, a, b\in \R
\end{equation}
(see, e.g., \cite[10(I)]{L}), when applied to \eqref{lhs} with $a:=u_h(x)-u_h(y)$ and $b:=u(x)-u(y)$, furnishes
$$\langle (-\Delta_p)^s u_h-(-\Delta_p)^s u, \delta_h u\rangle\geq c_p[\delta_h u]_{s,p}^p.$$
Through \eqref{eqdiff}--\eqref{fd2u}, this entails
\begin{equation}\label{newone}
\left[\delta_hu\right]_{s,p}^p\leq C\, \|f\|_r\, \|u\|_t^{1-\theta}\, \|\delta^2_h u\|_p^\theta.
\end{equation}
Since, by Lemma A1 of \cite{BLP},
\begin{equation}\label{BL}
\sup_{|h|>0}\left\|\frac{\delta_h^2 v}{|h|^\sigma}\right\|_p\leq 2\sup_{|h|>0}\left\|\frac{\delta_h v}{|h|^\sigma}\right\|_p\le C\, [v]_{\sigma, p}\quad \forall\, \sigma\in \ ]0, 1[, \quad p\geq 1,
\end{equation}
we have
$$\left\|\frac{\delta^2_hu}{|h|^{s+\frac{\theta\beta}{p}}}\right\|_p
=\frac{1}{|h|^{\frac{\theta\beta}{p}}}\left\|\frac{\delta_h(\delta_{-h} u)}{|h|^s}\right\|_p
\leq \frac{1}{|h|^{\frac{\theta\beta}{p}}}\sup_{|k|>0}\left\|\frac{\delta_{k}(\delta_{-h}u)}{|k|^s}\right\|_p
\leq \frac{C}{|h|^{\frac{\theta\beta}{p}}}[\delta_{-h}u]_{s,p}
=C\left[\frac{\delta_h u}{|h|^\frac{\theta\beta}{p}}\right]_{s,p},$$
which, on account of \eqref{newone}, easily leads to
$$\left\|\frac{\delta^2_hu}{|h|^{s+\frac{\theta}{p}\beta}}\right\|_p^p\leq C\left[\frac{\delta_h u}{|h|^\frac{\theta\beta}{p}}\right]_{s,p}^p\leq C\, \|f\|_r\, \|u\|_t^{1-\theta}\, \left\|\frac{\delta^2_h u}{|h|^{\beta}}\right\|_p^{\theta}$$
for any fixed $h\neq 0$, $\beta>0$. If $\beta:=s$ then the right-hand side is finite, because $u\in \dot{W}^{s,p}(\R^N)$ and \eqref{BL} holds. We can thus iterate on the differentiability orders $\beta_n$ defined as
\[
\begin{cases}
\beta_0:=s,\\
\beta_{n+1}:=s+\frac{\theta}{p}\beta_n, 
\end{cases}
\quad\Rightarrow\quad  \lim_{n\to+\infty}\beta_n=\frac{s}{1-\frac{\theta}{p}}=:\beta_\infty,
\]
producing the inequality
\[
\left\|\frac{\delta^2_hu}{|h|^{\beta_{n}}}\right\|_p\leq \left(C\, \|f\|_r\, \|u\|_t^{1-\theta}\right)^{\frac{1}{p}\sum_{i=0}^{n-1}\frac{\theta^i}{p^i}}\left\|\frac{\delta^2_h u}{|h|^{\beta_0}}\right\|^{\frac{\theta^{n}}{p^{n}}}_p,\quad n\in\N.
\]
Since $\theta/p<1$, one arrives at
\[
\left\|\frac{\delta^2_hu}{|h|^{\beta_\infty}}\right\|_p=\lim_{n\to +\infty}\left\|\frac{\delta^2_hu}{|h|^{\beta_n}}\right\|_p\leq C\, \|f\|_r^{\frac{1}{p-\theta}}\, \|u\|_t^{\frac{1-\theta}{p-\theta}},
\]
and \eqref{Bp>2} follows (recall that the previous inequality holds for a.e. $h\neq 0$).

{\em Case 2: $1<p<2$}.\\
It is known that \eqref{dp1} no longer holds. Nevertheless,
\begin{equation}\label{dp2}
(a^{p-1}-b^{p-1})(a-b)\geq c_p\frac{|a-b|^2}{(a^2+b^2)^{\frac{2-p}{2}}}\quad\forall\, (a,b)\in\R^2\setminus\{(0,0)\};
\end{equation}
cf. \cite[Lemma B.4]{BP}. Setting $a:=u_h(x)-u_h(y)$, $b:=u(x)-u(y)$, and raising \eqref{dp2} to $p/2$,  we obtain
\[
\begin{split}
|\delta_hu(x)-\delta_h u(y)|^p\leq &c_p^{\frac{p}{2}}\left[\left( (u_h(x)-u_h(y))^{p-1}-(u(x)-u(y))^{p-1}\right)(\delta_hu(x)-\delta_hu(y))\right]^{\frac{p}{2}}\times\\
&\times \left[ |u_h(x)-u_h(y)|^{2}+|u(x)-u(y)|^{2}\right]^{\frac{2-p}{2}\frac{p}{2}}.
\end{split}
\]
Next, multiply by $|x-y|^{-N-ps}$, integrate over $\R^{2N}$, and apply  H\"older's inequality with exponents $\frac{2}{p}$,
$\frac{2}{2-p}$. Thanks to \eqref{lhs}, this entails
\[
[\delta_h u]^p_{s,p}\leq C\langle (-\Delta_p)^s u_h-(-\Delta_p)^s u, \delta_h u\rangle^{\frac{p}{2}}\left(\int_{\R^{2N}}\frac{\left( |u_h(x)-u_h(y)|^{2}+|u(x)-u(y)|^{2}\right)^{\frac{p}{2}}}{|x-y|^{N+ps}}dxdy\right)^{1-\frac{p}{2}},
\]
which, through \eqref{eqdiff}, \eqref{fd2u}, besides the sub-additivity of $\tau\mapsto |\tau|^{\frac{p}{2}}$, gives
\[
[\delta_h u]^p_{s,p}\leq C\left(\|f\|_r\, \|\delta^2_h u\|_{r'}\right)^{\frac{p}{2}}\left([u_h]_{s,p}^p+[u]_{s,p}^p\right)^{1-\frac{p}{2}}\leq C\, [u]_{s,p}^{p(1-\frac{p}{2})}\left(\|f\|_r\, \|u\|_t^{1-\theta}\, \|\delta^2_h u\|_{p}^\theta\right)^{\frac{p}{2}}.
\]
Pick $\beta>0$ and divide by $|h|^{\beta\theta\frac{p}{2}}$. Like before we have
\[
\left\|\frac{\delta^2_h u}{|h|^{s+\beta\frac{\theta}{2}}}\right\|_p\leq C\, [u]_{s,p}^{(1-\frac{p}{2})}\left(\|f\|_r\, \|u\|_t^{1-\theta}\right)^{\frac{1}{2}}\left\|\frac{\delta^2_h u}{|h|^{\beta}}\right\|_{p}^{\frac{\theta}{2}}.
\]
Let us finally iterate on the differentiability orders $\beta_n$ defined as
\[
\begin{cases}
\beta_0:=s,\\
\beta_{n+1}:=s+\frac{\theta}{2}\beta_n
\end{cases}
\quad\Rightarrow\quad\lim_{n\to +\infty}\beta_n= \frac{2s}{2-\theta},
\]
to achieve the inequality
\[
\left\|\frac{\delta^2_h u}{|h|^{\frac{2s}{2-\theta}}}\right\|_p\leq \left[C\, [u]_{s,p}^{(1-\frac{p}{2})}\left(\|f\|_r\, \|u\|_t^{1-\theta}\right)^{\frac{1}{2}}\right]^{\sum_{i=0}^{+\infty}\frac{\theta^i}{2^i}},
\]
valid for a.e. $h\neq 0$, whence \eqref{Bp<2} follows after an elementary calculation.
\end{proof}
\begin{remark}
Estimates \eqref{Bp>2}--\eqref{Bp<2} can naturally be re-casted in the framework of Besov spaces.   Putting
\begin{equation}\label{defsigma}
\sigma:=
\begin{cases}
s\frac{p}{p-\theta}&\text{if $p\geq 2$},\\
s\frac{2}{2-\theta}&\text{if $1<p<2$},
\end{cases}
\end{equation}
the conditions $s\in \ ]0, 1[$, $\theta\in\  ]0, 1]$ force $\sigma\in\  ]s, 2s]\subseteq \ ]0, 2[$. So, the left-hand sides of \eqref{Bp>2}--\eqref{Bp<2} read as $[u]_{B^{\sigma}_{p, \infty}}$. Further, when $r<\frac{N}{ps}$ and $2\leq p\leq r'<p^*$, combining \eqref{stimalr1} with \eqref{Bp>2} easily yields
$$[u]_{B^{\sigma}_{p, \infty}}\leq C\, \|f\|_r^{\frac{1}{p-1}}.$$
\end{remark}
\section{Decay estimates}
We are now ready to prove the pointwise and Sobolev estimates stated in Section 1. 
\begin{lemma}[Interpolation inequality]\label{intbesov}
Let $p>1>\tau>s>0$, $\gamma\in\ ]0, p[$, and $\mu\in \ ]0, 1[$. Then there exists a constant  $C:=C(N,p,s,\gamma,\mu,\tau)>0$ such that
$$[u]_{s, \gamma}\leq C\, R^{\frac{N}{\gamma}-\frac{N}{p}+\mu(\tau-s)}\, [u]_{B^{\tau}_{p, \infty}}^\mu\, [u]_{s,p}^{1-\mu}$$
for every $u\in \dot B^{\tau}_{p,\infty}(\R^N)\cap \dot W^{s,p}(\R^N)$ with ${\rm supp} (u)\subseteq B_R$.
\end{lemma}

\begin{proof}
Suppose $R=1$. Observe that if $|h|>2$ then $u$, $u_h$, and $u_{2h}$ have disjoint supports. Hence,
\begin{equation}\label{ugu}
{\rm supp}(u)\subseteq B_1\quad \Rightarrow \quad\|\delta^2_hu\|_q=2^{1/q} \|\delta_h u\|_q=4^{1/q}\|u\|_q\quad \text{for any $|h|>2$, $q>0$,}
\end{equation}
which implies
\[
\begin{split}
[u]_{s,\gamma}^\gamma&=\int_{|h|\le 2}\frac{\|\delta_h u\|_\gamma^\gamma}{|h|^{s\gamma}}\frac{dh}{|h|^N}+2\int_{|h|> 2}\frac{\|u\|_\gamma^\gamma}{|h|^{s\gamma}}\frac{dh}{|h|^N}=\int_{|h|\le 2}\frac{\|\delta_h u\|_\gamma^\gamma}{|h|^{s\gamma}}\frac{dh}{|h|^N}+C_1\, \|u\|_\gamma^\gamma\\
&\leq \int_{|h|\le 2}\frac{\|\delta_h u\|_\gamma^\gamma}{|h|^{s\gamma}}\frac{dh}{|h|^N}+C_2\, \|u\|_p^\gamma.
\end{split}
\]
The first term will be estimated through successive applications of the H\"older's inequality
\[
\begin{split}
\int_{|h|\leq 2}\frac{\|\delta_h u\|_\gamma^\gamma}{|h|^{s\gamma}}\frac{dh}{|h|^N}&\leq 
C_3\int_{|h|\leq 2}\frac{\|\delta_h u\|_p^\gamma}{|h|^{s\gamma}}\frac{dh}{|h|^N}\\
&=C_3\int_{|h|\leq 2}\frac{\|\delta_h u\|_p^{(1-\mu)\gamma}}{|h|^{s(1-\mu)\gamma}}\frac{\|\delta_h u\|_p^{\mu\gamma}}{|h|^{\tau\mu\gamma}}\frac{dh}{|h|^{N-(\tau-s)\mu\gamma}}\\
&\leq C_3\, [u]_{B^{\tau}_{p,\infty}}^{\mu\gamma}\int_{|h|\leq 2}\frac{\|\delta_h u\|_p^{(1-\mu)\gamma}}{|h|^{s(1-\mu)\gamma}}|h|^{(\tau-s)\mu\gamma}\frac{dh}{|h|^{N}}\\
&\leq C_3\, [u]_{B^{\tau}_{p,\infty}}^{\mu\gamma}\left(\int_{|h|\leq 2}\frac{\|\delta_h u\|_p^{p}}{|h|^{sp}}\frac{dh}{|h|^{N}}\right)^{\frac{(1-\mu)\gamma}{p}}\left(\int_{|h|\leq 2}|h|^{(\tau-s)\mu\gamma}\frac{dh}{|h|^{N}}\right)^{1-\frac{(1-\mu)\gamma}{p}}\\
&\leq C_4\, [u]_{B^{\tau}_{p, \infty}}^{\mu\gamma}[u]_{s,p}^{(1-\mu)\gamma}.
\end{split}
\]
To evaluate the other term we  use \eqref{ugu} and obtain
\[
\|u\|_p^\gamma=C_5\, \|u\|_p^\gamma\left(\int_{|h|>2}\frac{dh}{|h|^{N+ps}}\right)^{\frac{\gamma}{p}}=\frac{C_5}{2^{\gamma/p}}\left(\int_{|h|>2}\frac{\|\delta_hu\|_p^p}{|h|^{ps}}\frac{dh}{|h|^{N}}\right)^{\frac{\gamma}{p}}\leq \frac{C_6}{2^{\gamma/p}}[u]_{s,p}^\gamma.
\]
Similarly, by H\"older's inequality and \eqref{ugu} again,
\[
\|u\|_p^\gamma=2^{\tau\gamma}\sup_{|h|>2} \left(\frac{\|u\|_p}{|h|^\tau}\right)^\gamma=\frac{2^{\tau\gamma}}{4^{\gamma/p}} \left(\sup_{|h|>2}\left\|\frac{\delta_{h}^2u}{|h|^\tau}\right\|_p \right)^\gamma\leq \frac{2^{\tau\gamma}}{4^{\gamma/p}}[u]_{B^{\tau}_{p,\infty}}^\gamma.
\]
Gathering together the above inequalities yields
\[
\|u\|_p^\gamma\leq C_7\, [u]_{B^{\tau}_{p, \infty}}^\mu[u]_{s,p}^{1-\mu},
\]
as desired. Now, the general case $R\neq 1$ comes out from a standard scaling argument.
\end{proof}
\begin{remark}
The conclusion of Lemma \ref{intbesov} actually holds for any $\tau\in\ ]0,p[$, but the proof is slightly more complicated once $\tau\geq 1$, which we do not need here. Moreover, it should be noted that the constant $C$ blows up as $\tau\to s^+$, because $C\geq C_4$ and 
$$C_4:=C_3\left(\int_{|h|\leq 2}|h|^{(\tau-s)\mu\gamma}\frac{dh}{|h|^{N}}\right)^{1-\frac{(1-\mu)\gamma}{p}}.$$
This is quite natural, since otherwise one would obtain the limiting inequality 
\[
[u]_{s,\gamma}\leq C\, [u]_{B^s_{p, \infty}}^\mu[u]_{s,p}^{1-\mu},\quad {\rm supp}(u)\subseteq B_1 
\]
which, when combined with \eqref{BL}, would imply the embedding 
$\dot{W}^{s,p}(B_1)\hookrightarrow \dot{W}^{s,\gamma}(B_1)$ for all $\gamma<p$. However, this embedding is false as soon as $s\in\ ]0, 1[$ and $1\leq \gamma <p$; cf. \cite{MiS}.
\end{remark}
\begin{remark}
If $\gamma>\frac{N}{N+s}$ then the interpolation inequality above has a very simple proof. In fact, in such a case, the Sobolev space $W^{s,\gamma}(B_1)$ coincides with the Besov space $B^{s}_{\gamma, \gamma}(B_1)$, for which a complete interpolation theory is available. In particular, since $\tau>s$, \cite[Theorem 3.3.6, iii)]{T} gives 
\[
\left(B^{\tau}_{p, \infty}(B_1); B^s_{p, p}(B_1)\right)_{\mu, \gamma}=B^{\mu \tau+(1-\mu) s}_{p, \gamma}(B_1)
\]
with $(X; Y)_{\mu, \gamma}$ denoting the Lions-Peetre real interpolation space. Hence, 
\[
\|u\|_{B^{\mu \tau+(1-\mu) s}_{p, \gamma}(B_1)}\leq C\, \|u\|_{B^{\tau}_{p, \infty}(B_1)}^\mu\|u\|_{B^{s}_{p, p}(B_1)}^{1-\mu},
\]
where $\|\cdot\|_{B^{\sigma}_{q, r}(B_1)}=[\cdot]_{B^{\sigma}_{q, r}(B_1)}+\|\cdot\|_q$. On the other hand, classical embedding theorems \cite[Theorem 3.3.1, i]{T} yield $B^{\mu \tau+(1-\mu) s}_{p, \gamma}(B_1)\hookrightarrow B_{\gamma, \gamma}^s(B_1)$ because $\tau>s$ and $p>\gamma$, whence one readily infers the interpolation inequality
\[
\|u\|_{B_{\gamma, \gamma}^s(B_1)}\leq C\, \|u\|_{B^{\tau}_{p, \infty}(B_1)}^\mu\|u\|_{B^{s}_{p, p}(B_1)}^{1-\mu}, \quad \tau>s, \ \mu\in\  ]0, 1[.
\]
For us, the issue with this proof is twofold: not only one would need a homogeneous version of the previous inequality, but, more substantially, we will work with values of $\gamma$ which can be smaller than the threshold $\frac{N}{N+s}$.
\end{remark}

\begin{lemma}\label{Glemma}
Suppose $p>1>s>0$, $N>ps$, $\sigma\in \ ]s, 2[$, $r\in[1,\frac{p^*}{p^*-1}]$. If $u\in \dot W^{s,p}(\R^N)\cap \dot B^\sigma_{p, \infty}(\R^N)$ is a radially non-increasing weak solution of $(-\Delta_p)^su=f$, where $f\in L^r(\R^N)$, then
\[
[u]_{s,\gamma}<+\infty \quad \text{for every $\gamma\in \ ]r^*(p-1), p]$}.
\]
\end{lemma}
\begin{proof} 
Fix $\tau\in\  ]s,\min\{\sigma, 1\}[$ and $\lambda\in \ ]0, 1[$ such that $\tau=\lambda s+(1-\lambda)\sigma$. Since
$$\frac{\|\delta^2_hu\|_p}{|h|^\tau}=\frac{\|\delta^2_hu\|_p^\lambda}{|h|^{\lambda s}}\, \frac{\|\delta^2_hu\|_p^{1-\lambda}}{|h|^{(1-\lambda)\sigma}},$$
the elementary interpolation inequality 
$$[u]_{B^{\tau}_{p, \infty}}\leq [u]_{B^{s}_{p, \infty}}^\lambda[u]_{B^\sigma_{p, \infty}}^{1-\lambda}$$
holds. Thus, on account of \eqref{BL},
\begin{equation}\label{hh}
[u]_{B^{\tau}_{p, \infty}}\leq C\, [u]_{s, p}^\lambda\, [u]_{B^\sigma_{p, \infty}}^{1-\lambda}.
\end{equation}
As already pointed out, $N>ps$ forces $\frac{p^*}{p^*-1}<\frac{N}{ps}$, whence $r<\frac{N}{ps}$. So, due to  \eqref{decay}, we have $\lim_{t\to +\infty} u(t)=0$, where $u(x)=u(|x|)$ by abuse of notation. Now, consider the horizontal dyadic layer cake decomposition  
\begin{equation}\label{hdlcd}
u_0(t):=(u(t)-u(1))_+,\quad u_i(t):=\min\{u(2^{i-1})-u(2^i), (u(t)-u(2^{i}))_+\},\quad i\geq 1,
\end{equation}
of $u$. Setting $A_0:=[0, 1[$ and $A_i:=[2^{i-1}, 2^{i}[$, one can write, whenever $t\in A_k$ for some $k\geq 0$,
\[
u_i(t)=
\begin{cases}
u(2^{i-1})-u(2^{i})&\text{if $i>k$},\\
u(t)-u(2^k)&\text{if $i=k$},\\
0&\text{if $i<k$},
\end{cases}
\]
which means
\begin{equation}\label{HL}
u(t)= u(t)-u(2^k)+\sum_{i=k+1}^{+\infty}[u(2^{i-1})-u(2^{i})]=\sum_{i=0}^{+\infty} u_i(t).
\end{equation}
The above series converges in $L^\infty(\R^N)$ because
\[
\|u_i\|_\infty\le u(2^{i-1})\le \frac{C}{2^{b(i-1)}}\|f\|_r^{\frac{1}{p-1}},\quad\mbox{with}\quad
 b:=\frac{N}{p-1}\left(\frac 1 r-\frac{ps}{N}\right)>0,
\]
due to \eqref{decay} and the monotonicity of $u$. Observe next that $u_i=g_i\circ u$ for some $1$-Lipschitz continuous function $g_i$. Hence, $|\delta_h u_i|\leq |\delta_h u|$ and using \eqref{BL} in \eqref{hh} produces
\[
 [u_i]_{B^{\tau}_{p, \infty}}\leq C\,  [u]_{s, p}^\lambda\, [u]_{B^\sigma_{p, \infty}}^{1-\lambda}.
\]
By Lemma \ref{intbesov}, for every $\mu\in \ ]0, 1[$ there exists a constant $C_\mu=C(N,p,s,\gamma,\mu,\tau)>0$ fulfilling
\[
[u_i]_{s,\gamma}\leq C_\mu\,  2^{i\left(\frac{N}{\gamma}-\frac{N}{p}+\mu(\tau-s)\right)}[u_i]_{B^\tau_{p, \infty}}^\mu[u_i]_{s,p}^{1-\mu}.
\]
Therefore,
\begin{equation}\label{ui}
[u_i]_{s,\gamma}\leq C_\mu\,  2^{i\left(\frac{N}{\gamma}-\frac{N}{p}+\mu(\tau-s)\right)}[u]_{s, p}^{\lambda\mu}\, [u]_{B^\sigma_{p, \infty}}^{\mu(1-\lambda)}\, [u_i]_{s, p}^{1-\mu}.
\end{equation}
Since $u$ is radially non-increasing, $u_i\in W^{s,p}_0(B_{2^i})\cap L^\infty(\R^N)$, $i\geq 0$, namely $u_i$ turns out to be a suitable test function. Via Lemma \ref{lemmag}, besides the properties of $u_i$, we thus obtain
\[
[u_i]_{s,p}^p\leq \langle (-\Delta_p)^s u, u_i\rangle=\int_{\R^N} f\, u_i\, dx\leq\|f\|_{r}\,  u(2^{i-1})\, \omega_N^{\frac{1}{r'}}(2^i)^{\frac{N}{r'}}.
\]
Using Corollary \ref{cordecay}, this entails, 
\[
[u_i]_{s,p}^p\leq C_\mu\, \|f\|_r^{p'}(2^{i})^{\frac{N}{r'}+\frac{N}{p-1}(\frac{ps}{N}-\frac{1}{r})}, 
\]
which, when inserted into \eqref{ui}, gives
\begin{equation}\label{ui2}
[u_i]_{s,\gamma}\leq C_\mu\, \|f\|_r^{\frac{1}{p-1}}\, [u]_{s, p}^{\lambda\mu}\, [u]_{B^\sigma_{p, \infty}}^{\mu(1-\lambda)}\, 2^{iNa_\mu},
\end{equation}
where, to avoid cumbersome formulas,
\[
a_\mu:=a(p, s, \gamma, r, \tau, \mu):=\frac{1}{\gamma}-\frac{1}{p}+\mu(\tau-s)+\frac{1-\mu}{p}\left(\frac{1}{r'}+\frac{1}{p-1}(\frac{ps}{N}-\frac{1}{r})\right).
\]
Finally, since 
\[
\gamma>r^*(p-1)=\frac{Nr(p-1)}{N-sr}\quad \Leftrightarrow \quad a_0=\frac{1}{\gamma}-\frac{1}{p}+\frac{1}{p}\left(\frac{1}{r'}+\frac{1}{p-1}(\frac{ps}{N}-\frac{1}{r})\right)<0
\]
we can find a sufficiently small $\mu>0$ such that $a_\mu<0$. If $\gamma\geq 1$ then \eqref{HL}, the triangle inequality, and \eqref{ui2} yield
\[
[u]_{s,\gamma}\leq C_\mu\, \|f\|_r^{\frac{1}{p-1}}\, [u]_{s, p}^{\lambda\mu}\, [u]_{B^\sigma_{p, \infty}}^{\mu(1-\lambda)}\sum_{i=0}^{+\infty}2^{iNa_\mu}<+\infty,
\]
as desired. So, suppose $r^*(p-1)<1$ and $\gamma \in\  ]r^*(p-1), 1[$. From $\sum_{i=0}^{+\infty} u_i=u$ a.e. in $\R^N$ it follows
\[
\lim_{n\to +\infty} \frac{\left|\sum_{i=0}^{n} u_i(x)-\sum_{i=0}^{n} u_i(y)\right|^\gamma}{|x-y|^{N+\gamma s}}=\frac{|u(x)-u(y)|^\gamma}{|x-y|^{N+\gamma s}}
\]
for almost all $(x, y)\in \R^{2N}$. Now, Fatou's lemma and the subadditivity of $r\mapsto r^\gamma$ lead to
\[
[u]_{s, \gamma}^\gamma\leq \liminf_{n\to +\infty}\int_{\R^{2N}} \frac{\left|\sum_{i=0}^{n} u_i(x)-\sum_{i=0}^{n} u_i(y)\right|^\gamma}{|x-y|^{N+\gamma s}}\, dx\, dy\leq \lim_{n\to +\infty}\sum_{i=0}^n[u_i]_{s,\gamma}^\gamma\, ,
\]
and one can conclude as before using \eqref{ui2}. This completes the proof. 
\end{proof}
Theorem \ref{MT} will be a consequence of Lemmas \ref{Rlemma}, \ref{Glemma}, and the next two.
\begin{lemma}\label{luno}
Let $N>ps$, $q>p$, and $\alpha\in[0,ps[$ satisfy \eqref{scalingrelation}. Let moreover $u\in\dot W^{s,p}(\R^N)$ be a nonnegative minimizer of \eqref{I} with $\lambda:=1$. Then $u$ is radially non-increasing around some point and for appropriate $f\in L^1(\R^N)$ one has $(-\Delta_p)^s u=f$ weakly as per \eqref{locweak}.
\end{lemma}
\begin{proof}
We first stress that such an $u$ exists due to Theorem  \ref{exist}. Moreover, by Lemma \ref{radiality}, it turns out to be radially non-increasing around some point while standard arguments, chiefly based on \eqref{I}, yield
\begin{equation}\label{wf}
\langle (-\Delta_p)^s u, v\rangle= I_1 \int_{\R^N}\frac{u^{q-1}}{|x|^\alpha}\, v\, dx\quad \forall\, v\in \dot{W}^{s,p}(\R^N),
\end{equation}
where the integral at the right-hand side is absolutely convergent because of H\"older's and Hardy-Sobolev's inequalities. The conclusion will be achieved once one verifies that $x\mapsto u(x)^{q-1}/|x|^{\alpha}$ lies in $L^1(\R^N)$. Set, for any $\eps>0$,
\[
\psi_\eps(t):=\int_0^t \left[(\eps+\tau)^{-\frac{1}{q}}-\frac{1}{q}\tau(\eps+\tau)^{-1-\frac{1}{q}}\right]^p\,d\tau,
\quad t\in\R^+_0;
\]
cf. the proof of \cite[Proposition 3.3]{BMS}. The function $\psi_\eps:[0,+\infty)\to\R$ is Lipschitz continuous, increasing, and fulfills
\begin{equation}\label{geps}
0\le \psi_\eps(t)\leq \int_0^t (\eps+\tau)^{-\frac{p}{q}}\,d\tau=\frac{1}{1-\frac{p}{q}}((\eps+t)^{1-\frac{p}{q}}-\eps^{1-\frac{p}{q}})\leq \frac{1}{1-\frac{p}{q}}(\eps+t)^{-\frac{p}{q}}t.
\end{equation}
Further, if
\[
\Psi_\eps(t):=\int_0^t \psi'_\eps(\tau)^\frac{1}{p}\,d\tau=\frac{t}{(\eps+t)^{\frac{1}{q}}}
\]
then \eqref{wf}, written with $v:=\psi_\eps\circ u\in \dot{W}^{s,p}(\mathbb{R}^N)$, and Lemma \ref{lemmag} entail
\[
[\Psi_\eps(u)]_{s,p}^p\leq \langle (-\Delta_p)^s u, \psi_\eps\circ u\rangle= I_1\int_{\R^N}
\frac{u^{q-1}\psi_\eps\circ u}{|x|^\alpha}\, dx.
\]
Using \eqref{HS} on the left-hand term and \eqref{geps} on the right-hand one, we arrive at
\begin{equation}\label{lhrhs}
\int_{\R^N}\frac{u^q}{u+\eps}\frac{dx}{|x|^\alpha}\leq C_1\left(\int_{\R^N}\frac{u^{q-1}\psi_\eps\circ u}{|x|^\alpha}\, dx\right)^{\frac{q}{p}}\leq C_2\left(\int_{\R^N}u^{q-p}\frac{u^p}{(u+\eps)^{\frac{p}{q}}}\frac{dx}{|x|^\alpha}\right)^{\frac{q}{p}}.
\end{equation}
Observe next that to every $\delta>0$ there corresponds $K>0$ satisfying
\[
\int_{\{u<K\}}u^{q}\frac{dx}{|x|^\alpha}<\delta.
\]
Consequently, by H\"older's inequality, 
\begin{equation}\label{bla}
\begin{split}
\int_{\R^N}u^{q-p}\frac{u^p}{(u+\eps)^{\frac{p}{q}}}&\frac{dx}{|x|^\alpha}=\int_{\{u\geq K\}}u^{q-p}
\frac{u^p}{(u+\eps)^{\frac{p}{q}}}\frac{dx}{|x|^\alpha}+\int_{\{u<K\}}u^{q-p}\frac{u^p}{(u+\eps)^{\frac{p}{q}}}
\frac{dx}{|x|^\alpha}\\
&\leq \int_{\{u\geq K\}}u^{q-\frac{p}{q}}\frac{dx}{|x|^\alpha}+\left(\int_{\{u<K\}}u^{q}\frac{dx}{|x|^\alpha}\right)^{1-\frac{p}{q}}\left(\int_{\R^N}\frac{u^q}{u+\eps}\frac{dx}{|x|^\alpha}\right)^{\frac{p}{q}}\\
&\leq \int_{\{u\geq K\}}u^{q-\frac{p}{q}}\frac{dx}{|x|^\alpha}
+\delta^{1-\frac{p}{q}}\left(\int_{\R^N}\frac{u^q}{u+\eps}\frac{dx}{|x|^\alpha}\right)^{\frac{p}{q}}.
\end{split}
\end{equation}
From \eqref{lhrhs}--\eqref{bla} it now follows
\[
\int_{\R^N}\frac{u^q}{u+\eps}\frac{dx}{|x|^\alpha}
\leq C_3\left(\int_{\{u\geq K\}}u^{q-\frac{p}{q}}\frac{dx}{|x|^\alpha}\right)^{\frac{q}{p}}
+C_4\, \delta^{\frac{q}{p}-1}\int_{\R^N}\frac{u^q}{u+\eps}\frac{dx}{|x|^\alpha},
\]
whence 
\[
\int_{\R^N}\frac{u^q}{u+\eps}\frac{dx}{|x|^\alpha}\leq C\left(\int_{\{u\geq K\}}u^{q-\frac{p}{q}}
\frac{dx}{|x|^\alpha}\right)^{\frac{q}{p}}
\]
provided $\delta$ is sufficiently small. Here, $C:=C(p, s, \alpha)$ and $K:=K(p, s, \alpha, u)$. Since $u^{q-\frac{p}{q}}$ belongs to $ L^1_{\rm loc}(dx/|x|^\alpha)$, letting $\eps\to 0^+$ shows the claim. 
\end{proof}
\begin{lemma}\label{final}
Let $N>ps$, $q>p$, $\alpha\in[0,ps[$ satisfy \eqref{scalingrelation} and $u\in\dot W^{s,p}(\R^N)$ be a nonnegative minimizer of \eqref{I} with $\lambda:=1$. Then $u\in L^\infty(\R^N)$ and
$\frac{u^{q-1}}{|x|^\alpha}\in L^r(\R^N)$ for every $r\in [1, \frac{N}{\alpha}[$.
\end{lemma}
\begin{proof}
Given $k>0$, $t\geq 0$, and $\beta\geq 1$, we define $g_\beta(t):=t (t_k)^{\beta-1}$, where $t_k:=\min\{t,k\}$. An easy verification ensures that $g_\beta\circ u\in \dot W^{s,p}(\R^N)$ is a suitable test function in \eqref{wf}. Moreover, since $p\geq 1$, one obtains through elementary considerations
\[
G_\beta(t)\geq \frac{p\,\beta^{1/p}}{p+\beta-1}\, g_{\frac{\beta-1}{p}}(t),
\]
with $G_\beta$ as in \eqref{defG}. Exploiting Lemma \ref{lemmag} and Hardy-Sobolev's inequality entails
\[
C_\beta \left(\int_{\R^N} \frac{g_{\frac{\beta-1}{p}}(u)^q}{|x|^\alpha}\, dx\right)^{\frac{p}{q}}\leq  [G_\beta]_{s,p}^p\leq  \langle (-\Delta_p)^s u, g_\beta\circ u\rangle= \int_{\R^N}\frac{u^q\,  u_k^{\beta-1}}{|x|^\alpha}\, dx
\]
for some $C_\beta=C(N, p, s, \alpha, \beta)$, so that
\begin{equation}\label{mhs}
\left(\int_{\R^N} \frac{u^q\, u_k^{\frac{q}{p}(\beta-1)}}{|x|^\alpha}\, dx\right)^{\frac{p}{q}}\leq C_\beta\int_{\R^N}\frac{u^q\,  u_k^{\beta-1}}{|x|^\alpha}\, dx,
\end{equation}
for another $C_\beta=C(N, p, s, \alpha, \beta)$.
Observe then that for any $K>0$ one has
\[
\begin{split}
\int_{\R^N}\frac{u^q \, u_k^{\beta-1}}{|x|^\alpha}\, dx&\leq \int_{\{u<K\}}\frac{u^q \, u_k^{\beta-1}}{|x|^\alpha}\, dx+\int_{\{u\geq K\}}\frac{u^q\,  u_k^{\beta-1}}{|x|^\alpha}\, dx\\
&\leq K^{\beta-1}\int_{\R^N}\frac{u^q}{|x|^\alpha}\, dx+\left(\int_{\{u\geq K\}}\frac{u^q}{|x|^\alpha}\, dx\right)^{\frac{q-p}{q}}\left(\int_{\R^N}\frac{u^q\, u_k^{\frac{q}{p}(\beta-1)}}{|x|^\alpha}\, dx\right)^{\frac{p}{q}},
\end{split}
\]
where H\"older's inequality with respect to the measure $u^qdx/|x|^\alpha$ has been used on the second term. Since $u^q\in L^1(\R^N, dx/|x|^\alpha)$ and $q>p$, the last term can be reabsorbed on the left of \eqref{mhs} provided $K$ is large enough, thus arriving at
\[
\frac{1}{2} \left(\int_{\R^N} \frac{u^q\, u_k^{\frac{q}{p}(\beta-1)}}{|x|^\alpha}\, dx\right)^{\frac{p}{q}}\leq C_\beta\, K^{\beta-1}\int_{\R^N}\frac{u^q}{|x|^\alpha}\, dx.
\]
Now, let $k\to +\infty$. As $\beta\geq 1$ was arbitrary, we get $u\in L^t(\R^N, dx/|x|^\alpha)$ for all $t\geq q$. By \eqref{stimalr3}, the conclusion $u\in L^\infty(\R^N)$ is achieved once one verifies that $u^{q-1}/|x|^\alpha\in L^{\bar r}(\R^N)$ for some $\bar r>\frac{N}{ps}$. Hence, fix $R_0>1$ fulfilling $u\leq 1$ on $B_{R_0}^c$. Thanks to Lemma \ref{luno}, the function  $f(x):=u(x)^{q-1}/|x|^\alpha$ lies in $L^1(\R^N)$. Moreover, $0\leq f\leq 1$ on $B_{R_0}^c$, whence
\begin{equation}\label{poi}
\int_{B_{R_0}^c} f^\beta\, dx\leq \int_{B_{R_0}^c}f\, dx<+\infty\quad \forall\, \beta\geq 1.
\end{equation}
Finally, we choose  $\bar r\in\ ]\frac{N}{ps},\frac{N}{\alpha}[$ and $t>1$ so large that
\[ 
\alpha(\bar r-\frac{1}{t})t'<N
\]
so that H\"older's inequality with exponents $t$ and $t'$ yields
\[
\int_{B_{R_0}}f^{\bar r} dx=\int_{B_{R_0}}\frac{u^{\bar r(q-1)}}{|x|^{\frac{\alpha}{t}}}\frac{1}{|x|^{\alpha(\bar r-\frac{1}{t})}}\, dx
\leq \left(\int_{\R^N}\frac{u^{\bar r(q-1)t}}{|x|^\alpha}\, dx\right)^{\frac{1}{t}}
\left(\int_{B_{R_0}}\frac{1}{|x|^{\alpha(\bar r-\frac{1}{t})t'}}\, dx\right)^{1-\frac{1}{t}}.
\]
Since both integrals are finite, $f\in L^{\bar r}(\R^N)$, as desired. The remaining conclusion, namely $f\in L^r(\R^N)$ for every $r\in [1, \frac{N}{\alpha}[$, directly follows from the above inequality
and \eqref{poi}.
\end{proof}
\begin{proof}[Proof of Theorem \ref{MT}]
Theorem \ref{exist} provides a nonnegative, radially decreasing minimizer $u$ of \eqref{I}. Lemma \ref{luno} and Corollary \ref{cordecay} give the upper bound in \eqref{as1}, while the lower bound can be achieved via the same argument exploited to show \cite[Corollary 3.7]{BMS}. Indeed, by \cite[Corollary 5.5 and Theorem 5.2]{IMS}, $u$ is continuous on $\R^N\setminus \{0\}$ as well as everywhere positive. 

To get \eqref{as2}, we shall verify the hypotheses of Lemma \ref{Glemma} for $r=1$. The summability request $f:=(-\Delta_p)^s u\in L^1(\R^N)$ is stated in Lemma \ref{luno}. Concerning the regularity of $u$, observe that $N>ps>\alpha$ forces $\frac{p^*}{p^*-1}<\frac{N}{\alpha}$. Consequently, due to Lemma \ref{final}, $u\in L^\infty(\R^N)$ and $f\in L^{\frac{p^*}{p^*-1}}(\R^N)$. We can apply now Lemma \ref{regest}, with $r:=\frac{p^*}{p^*-1}$, $t:=+\infty$, to arrive at $u\in \dot B^\sigma_{p, \infty}(\R^N)$, where $\sigma\in \ ]s, 2[$ is given by \eqref{defsigma}. 

It remains to show that the function $U$ defined in \eqref{talentiane} satisfies \eqref{as1}--\eqref{as2}. Estimate \eqref{as1} is obvious, so we fix $\gamma>\frac{N(p-1)}{N-s}$ and prove \eqref{as2}. If $U_i$, $i\in\N_0$, denotes the horizontal dyadic layer cake decomposition (see, e.g., \eqref{hdlcd}) of $U$  then ${\rm supp}(U_i)\subseteq B_{2^i}$ and
\[
U_i\lfloor_{B_{2^{i-1}}}\equiv U(2^{i-1})-U(2^i)\leq C\, 2^{-i\frac{N-ps}{p-1}},\quad
{\rm Lip}(U_i)\leq C\,2^{-i(\frac{N-ps}{p-1}+1)}
\]
for appropriate $ C=C(p, s, \alpha)>0$. Moreover, $U=\sum_{i=0}^{+\infty} U_i$ pointwise. Using Lemma \ref{23} we obtain
\[
\||D^s U_i|^\gamma\|_\infty \le C_1 \, 2^{-i\gamma\frac{N-s}{p-1}},\qquad
|D^s U_i|^\gamma(x)\le C_2\,  \frac{2^{i(N-\gamma\frac{N-ps}{p-1})}}{|x|^{N+\gamma s}}\quad\forall\,x\in B_{2^{i+1}}^c,
\]
which entails
\[
\begin{split}
[U_i]_{s,\gamma}^\gamma&=\int_{B_{2^{i+1}}} |D^s U_i|^\gamma\, dx+\int_{B_{2^{i+1}}^c} |D^s U_i|^\gamma\, dx\\
&\leq C_1\, 2^{-i\gamma\frac{N-s}{p-1}}|B_{2^{i+1}}|
+C_2\, 2^{i(N-\gamma\frac{N-ps}{p-1})}\int_{B_{2^{i+1}}^c}\frac{dx}{|x|^{N+\gamma s}}\\
&\leq C_3\, 2^{i(N-\gamma\frac{N-s}{p-1})}+C_4\, 2^{i(N-\gamma\frac{N-ps}{p-1})}2^{-i\gamma s}
=C_5\, 2^{i(N-\gamma\frac{N-s}{p-1})}.
\end{split}
\]
Notice that the last exponent is negative, so that proceeding as in the final part of Lemma \ref{Glemma}'s proof gives the claim.
\end{proof}

\end{document}